\documentclass[11pt]{amsart}
\usepackage{amsmath,amssymb,epsf,eucal,amscd,latexsym,graphicx}

\DeclareMathOperator{\intr}{int}

\DeclareMathOperator{\id}{id}

\DeclareMathOperator{\supp}{supp}

\DeclareMathOperator{\tors}{torsion}

\DeclareMathOperator{\lip}{lipschitz}

\newcommand{\x}{\times}

\newcommand{\<}{\left<}
\renewcommand{\>}{\right>}
\renewcommand{\epsilon}{\varepsilon}
\renewcommand{\phi}{\varphi}

\newcommand{\bd}{\partial}

\newcommand{\ra}{\rightarrow} 

\newcommand{\hra}{\hookrightarrow} 
\newcommand{\ura}{\underrightarrow}
\newcommand{\xra}{\xrightarrow}
\newcommand{\ula}{\underleftarrow}
\newcommand{\ol}{\overline} 
\newcommand{\wh}{\widehat} 
\newcommand{\trb}{\bd_{\pitchfork}} 
\newcommand{\tr}{\pitchfork} 
\newcommand{\tb}{\bd_{\tau}}
\newcommand{\CO}{C^{0}} 
\newcommand{\COO}{C^{0+}}

\newcommand{\CI}{C^{1}}
\newcommand{\CII}{C^{2}}
\newcommand{\Ci}{C^{\infty}}

\newcommand{\SI}{S^{1}}

\newcommand{\CC}{\mathcal{C}} 
\newcommand{\FF}{\mathcal{F}}

\newcommand{\LL}{\mathcal{L}} 
\newcommand{\KK}{\mathcal{K}}

\newcommand{\EE}{\mathcal{E}} 
\newcommand{\II}{\mathcal{I}}

\newcommand{\BB}{\mathcal{B}} 
\newcommand{\HH}{\mathcal{H}} 
\newcommand{\GG}{\mathcal{G}} 
\newcommand{\UU}{\mathcal{U}} 

\newcommand{\DD}{\mathcal{D}}
\newcommand{\XX}{\mathcal{X}}
\newcommand{\ZZ}{\mathcal{Z}}

\newcommand{\T}{\mathfrak{T}} 
 
\newcommand{\C}{\mathfrak{C}} 
\newcommand{\Ck}{\C^{\kappa}}
\newcommand{\R}{\mathbb{R}}
\newcommand{\Z}{\mathbb{Z}}

\newcommand{\Varphi}{\Phi}
\newcommand{\wt}{\widetilde} 
\newcommand{\0}{\emptyset} 
\newcommand{\hk}{H^{1}_{\kappa}}
\newcommand{\hI}{H^{1}}
\newcommand{\hi}{H_{1}}
\newcommand{\sm}{\smallsetminus} 

\newtheorem{theorem}{Theorem}[section]
\newtheorem{lemma}[theorem]{Lemma}
\newtheorem{cor}[theorem]{Corollary}
\newtheorem{prop}[theorem]{Proposition}

\theoremstyle{definition} 
\newtheorem{defn}[theorem]{Definition} 
\newtheorem{rem}{Remark}

 \newtheorem{example}{Example}
\newtheorem{ass}{Assumption}

\begin{document}

\title{Open Saturated Sets Without Holonomy}

\author[J. Cantwell]{John Cantwell}
\address{Department of Mathematics\\ St. Louis University\\ St. Louis, MO 
63103}
\email{CANTWELLJC@SLU.EDU}

\author[L. Conlon]{Lawrence Conlon}
\address{Department of Mathematics\\ Washington University, St. Louis, MO 
63130}
\email{CONLONLAWRENCE@ICLOUD.COM}


\begin{abstract}
Open, connected, saturated sets $W$ without holonomy in codimension one foliations $\FF$ play key roles as fundamental building blocks.  A great deal is understood about the structure of $\FF|W$, but the possibilities have not been quantified.  Here, for the case of foliated 3-manifolds, we produce a finite system of closed, convex, non-overlapping polyhedral cones in $\hI(W)$ (real coefficients) such that the isotopy classes of possible foliations $\FF|W$ without holonomy correspond one-to-one to the rays in the interiors of these cones.  This   generalizes our  classification of depth one  foliations to foliations of finite depth and more general foliations.
\end{abstract}

\maketitle

\section{Introduction}

 W.~\!\!Thurston shows that fibrations $\pi:M \to \SI$, where $M$ is a compact $3$-manifold,  are in one-one corresponance, up to isotopy, with rays through the nontrivial lattice points in the open cones  over certain top dimensional faces  of the unit ball  of the Thurston norm in $\hI(M;\R)$.    In~\cite{cc:cone}, we produced closed polyhedral cones in  $\hI(M;\R)$, where $(M,\gamma)$ is a taut sutured 3-manifold, such that the depth-one  foliations of $M$  having the components of $R(\gamma)$ as sole compact leaves are in one-one corresponance, up to isotopy, with rays through the nontrivial lattice points (called rational rays) in the interiors of these  cones. The cones have disjoint interiors which are open, convex, polyhedral cones. In~\cite{cc:almostnohol}, we correct some errors in~\cite{cc:cone} and further show that the dense leaved foliations having the components of $R(\gamma)$ as sole compact leaves and sole leaves with holonomy are in one-one corresponance, up to isotopy, with the irrationsl rays in these open cones.

In this paper, we apply the methods of~\cite{cc:cone} and~\cite{cc:almostnohol}, suitably modified,  to  open, connected, $\FF$-saturated sets without holonomy in tautly foliated 3-manifolds $(M,\FF)$.   Typically a finite depth foliation has finitely many leaves with holonomy, the complement of which are finitely many open, connected, $\FF$-saturated sets without holonomy. Thus, our results extend the results of \cite{th:norm}, \cite{cc:cone}, and~\cite{cc:almostnohol} to  finite depth foliations of $3$-manifolds and, more generally, to other taut  foliations.
In particular, our results apply to some foliations with exceptional minimal sets. For example,

\begin{example}

Example~4.1.6 in~\cite{condel1} is a foliation $\FF$ of a compact $3$-manifold with a unique exceptional minimal set. The complement of this minimal set is an  open, connected, $\FF$-saturated set $W$ without holonomy to which our results apply. Note that every leaf in the set $W$ has a Cantor set of ends.

\end{example}

\section{Definitions, assumptions, and results}

Let $(M,\FF)$ be a compact, connected, transversely oriented, foliated $n$-manifold of codimension~$1$ and class $\COO$.  That is, the leaves are $\CI$ and the tangent planes to the leaves form a topological vector bundle over $M$. Less formally, $\FF$ is integral to a $\CO$ plane field. Such a foliation admits a transverse, $1$-dimensional foliation $\LL$ of class $\Ci$.  The manifold $M$ need not be oriented, hence leaves may fail to be orientable.  We assume throughout that $(M,\FF)$ is not a foliated product.

\begin{rem}
For foliations $\GG$ of arbitrary codimension~$q$ on $M$, there is always to be a decomposition $\bd M=\tb M\cup\trb M$, where the tangential boundary $\tb M$ is a union of leaves of $\GG$ and the transverse boundary $\trb M$ is transverse to $\GG$. Both, of course, are  $(n-1)$-dimensional. Thus, $\GG|\trb M$ is a foliation of leaf dimension $n-q-1$ and $\GG|\tb M$ is a foliation of leaf dimension $n-q$.  If $\tb M\cap\trb M\ne\0$, this intersection is to consist of corners which are convex with respect to $\GG$.  In this paper, $\tb M$ and $\trb M$ will always refer to the codimension~$1$ foliation $\FF$.  Thus, each component of $\tb M$ will be a compact leaf of $\FF$.  When we have cause to consider the tangential and transverse boundary for higher codimnsion foliations (typically $1$-dimensional foliations $\LL$) we will use the words, not the symbols $\tb M$ and $\trb M$.
\end{rem}

If $\dim M = 3$,  the transverse boundary $\trb M$ has  components  which are annuli, tori and/or Klein bottles. (M\"obius strips are ruled out by the transverse orientability of $\FF$.) We require that $\FF$ be \emph{taut}, meaning that each leaf meets either a closed transversal to $\FF$ or a transverse, properly imbedded arc with endpoints in $\tb M$. 
We also require that $\FF$ be \emph{boundary taut}, meaning that the induced $1$-dimensional foliation $\FF|\trb M$ is to have no $2$-dimensional Reeb components.  The annular components of $\trb M$ meet $\tb M$ in convex corners. Except for the possible nonorientability, $M$ is a sutured manifold in the sense of D.~Gabai~\cite{ga0}. In the standard notation for sutured manifolds $(M,\gamma)$, $\gamma=\trb M$ and $R(\gamma)=\tb M$.

\subsection{The transverse completion}

Let $W\subset M$ be an open, connected, $\FF$-saturated set. It will be useful to consider the transverse completion $\wh W$ of $W$.  This is also called the metric completion, being the completion of $W$ in any Riemannian metric  on $W$ relativized from a Riemannian metric on $M$.  For the properties of $\wh W$ summarized here, see~\cite[Section~5.2]{condel1}.  This completion adds finitely many boundary leaves to $W$.  Two such boundary leaves may be identical in $M$, but distinct as leaves in $\tb\wh W$.  Their union in $M$, denoted by $\delta W$, is called the border of $W$.   The extension $\wh\FF$ of $\FF|W$ to $\wh W$ is obtained by adding on the components of $\tb\wh W$ as leaves.  We again use the notation $\bd\wh W=\tb\wh W\cup\trb\wh W$.  When $\dim M=3$, $\trb \wh W\subset\trb M$  is a  union of annuli, tori, Klein bottles and/or (possibly infinitely many) infinite rectangular strips.  If $\dim M=3$ the tautness of $\FF$ implies that $\wh\FF$ is taut and boundary tautness for $\wh\FF$ follows from that of $\FF$.

\subsection{Foliations of $W$ without holonomy} 

We are interested in the case that $\FF|W$ is without holonomy.   If $\FF$ is of class $\CII$ (class $C^{1+\lip}$ suffices), there are only two possibilities, either $\FF|W$ fibers $W$ over $\SI$ or each leaf of $\FF|W$ is dense in $W$~\cite[Theorem~9.1.4]{condel1}. But if the smoothness class is $\CI$ or less, one could have exceptional minimal sets of Denjoy type.

Since, generally, we only require $\FF$ to be of class $\COO$, we place the following definition. 

\begin{defn}
If the foliation $\FF|W$  fibers $W$ over $\SI$, we say that $\wh\FF$ is of \emph{relative depth one}.  More generally, if $\FF|W$ is without holonomy and is either dense leaved or fibers $W$ over $\SI$, we say that $\wh\FF$ is \emph{relatively almost without holonomy}. 
\end{defn}

We will be interested in  varying the foliation of $W$ through a family of foliations $\HH$ which extend to foliations $\wh\HH$ of $\wh W$, relatively almost without holonomy.  Thus, $\HH$ extends $\FF|(M\sm W)$ to a $\COO$ foliation of $M$.  But even if $\FF$ is $\Ci$ and  $\wh\HH$ is $\Ci$, $\HH$ may only extend $\FF|(M\sm W)$ to a $\COO$ foliation of $M$.  

\subsection{The octopus decomposition}

The manifold $\wh W$ admits an ``octopus decomposition'' into a compact nucleus and finitely many arms that are foliated interval bundles~\cite[Definition~5.2.13]{condel1}.  In the case that $\FF$ is $\Ci$ and $\wh\FF$ is of relative depth one,  $\FF$ is a product foliation in the arms of an octopus decomposition with suitably large nucleus (compactness of the junctures, cf.~\cite[Theorem~8.1.26]{condel1}).  We say that the foliation is \emph{trivial in the arms}. 

We will need to require that the foliations $\wh\HH$ of $\wh W$ admit a smooth \emph{strongly transverse} $1$-dimensional foliation $\wh\LL$ in the following sense.

\begin{defn}
A $1$-dimensional foliation $\wh \LL$ of $\wh W$  is \emph{strongly transverse} to the codimension~$1$ foliation $\wh\HH$ of $\wh W$ (and vice versa) if 

\begin{enumerate}

\item $\wh\LL$ is of class $\Ci$\upn{;}\label{itemone}

\item   $\trb\wh W$ is the tangential boundary for $\wh\LL$ and  $\tb\wh W$ is the transverse boundary for $\wh\LL$\upn{;}\label{itemtwo} 

\item $\wh\LL$ and $\wh\HH$ are transverse\upn{;}\label{itemthree}

\item in the arms $A_{i}$ of some octopus decomposition of $\wh W$, $\wh\LL|A_{i}$ defines a trivial fibration by compact intervals.\label{itemfour}

\end{enumerate}

\end{defn}

\begin{rem}
Any $1$-dimensional foliation $\LL$ of $M$ which is  transverse to $\FF$ induces such a foliation $\wh\LL$ of $\wh W$ transverse to $\wh\FF$.  The main point of strong tranversality is that the structure of $\wh\LL$ in the arms of the octopus prevents an end of a leaf of $\wh\LL$ from ``scooting off'' to infinity without accumulating anywhere in $\wh W$.
\end{rem}

\subsection{Distinction between the $\COO$ case and the $\CII$ case}\label{distinctC2}

Remark that the inclusion $W\hra\wh W$ is a homotopy equivalence, hence $\hI(W)=\hI(\wh{W})$.  Let $\hk(\wh{W})\subset\hI(\wh{W})$ be the subspace of classes represented by compactly supported cocycles.  All of these spaces are generally infinite dimensional, but carry natural topologies.  Unless otherwise indicated, homology and cohomology are computed with real coefficients.

If $\FF$  is a $\CII$ foliation on $M$, it is well known that $\FF|W$ will be trivial (that is a product foliation) in the arms of some octopus decomposition of $W$. In this case we work in the cohomology group $\hk(\wh{W})$. In the $\COO$ case this is not true and we work in $\hI(\wh{W})$.

\subsection{Proof of main results}

The proof of our main results involves three very different parts. The first part is developed in Section~\ref{SchSul}, uses the Schwartzmann-Sullivan theory of  asymptotic cycles and is valid for  $\dim M\ge 3$ with the above assumptions on $M$.   The second part invloves the Handel-Miller theory of endperiodic automorphisms which is only valid for $\dim M = 3$, is developed in Sections~\ref{HMcones} and~\ref{HMcones1}, and requires additional restrictions on the topology of the leaves. The third part we leave to a separate paper~\cite{cc:isorel} (see the remark after the statements of Theorems~\ref{cone} and~\ref{conesmooth}).

\subsection{Restrictions on the topology of the leaves}\label{restrict}

In Sections~\ref{HMcones} and~\ref{HMcones1} we require that $\dim M=3$. We further require  that \textbf{all compact leaves of $\FF$ have strictly negative Euler characteristic}.  Disks, spheres and projective planes are ruled out due to Reeb stability.  Annuli, M\"obius strips, tori and Klein bottles are ruled out to guarantee the following.  
\begin{enumerate}
\item The foliation is taut. 
\item If the foliation is smooth of class  $\Ci$, no semiproper leaf  has a simple end.
\end{enumerate}
A simple end is one with a neighborhood homeomorphic to $J\x[0,\infty)$, $J=I$ or $J=\SI$.  The fact that Property~(1) follows from the restriction on the leaves is essentially a theorem of S.~Goodman~\cite{sue:closed} and~\cite[Theorem~6.3.5]{condel1}, extended in a fairly straightforward way  to manifolds with boundary and corners. Property~(2) will be proven in this paper (Proposition~\ref{nosimple}).   For foliations of class $\COO$, we need in addition  the following very weak restriction on the topology of the leaves.

\begin{ass}
No semiproper leaf is contractible nor has the homotopy type of the circle.
\end{ass}

For $\Ci$ foliations, this is automatic by (2).   All of this is needed for the Handel-Miller theory of endperiodic automorphisms of surfaces~\cite{cc:hm}.

\subsection{The main results}

We assume that $\dim M = 3$ and that $M$ satisfies all the above assumptions. We make the following definition to simplify the statements of Theorems~\ref{cone} and~\ref{conesmooth}.

\begin{defn}
A $\COO$ foliation $\wh\HH$ of $\wh W$  is \emph{admissible} if it admits a strongly transverse $1$-dimensional foliation  $\wh\LL$ and is either of relative depth~$1$ or is dense leaved without holonomy in $W$.    
\end{defn}

Here is our main result in the $\COO$ case. The proof is given in Section~\ref{HMcones1}.

\begin{theorem}\label{cone}
There is a finite set of   closed, convex, polyhedral cones $\C_{1},\C_{2},\dots,\C_{r}\subset\hI(\wh{W})$ with disjoint interiors, each with finitely many faces of codimension one,  such that the $\CO$ isotopy classes  of admissible foliations $\wh\HH$ of $\wh W$  are in natural one-to-one correspondence with the rays out of the origin in the interiors of these cones. 
\end{theorem}

Here is our main result in the $\CII$ case (see Section~\ref{distinctC2}). The proof is given in Section~\ref{HMcones}.

\begin{theorem}\label{conesmooth} 
There is a finite set of  closed, convex, polyhedral cones $\C^{\kappa}_{1},\C^{\kappa}_{2},\dots,\C^{\kappa}_{r}\subset\hk(\wh{W})$  with disjoint interiors, each with finitely many faces of codimension one,  such that the $\CO$ isotopy classes of admissible foliations $\wh\HH$ of $\wh W$ that are  trivial in the arms are in natural one-to-one correspondence with the rays out of the origin in the interiors of these cones.   
\end{theorem}

 \begin{defn}\label{folnconedef}
The rays in the interiors of these cones will be called \emph{foliated rays}. The foliated rays meeting  nontrivial elements of the integer lattice of the cohomology groups are called \emph{rational foliated rays}.  The foliated rays not meeting any  nontrivial elements of the integer lattice of the cohomology groups are called \emph{irrational foliated rays}.
\end{defn}

\begin{rem}
The rational rays  correspond to the foliations of relative depth one. The irrational rays  correspond to the foliations that are dense leaved without holonomy in $W$.     
\end{rem}

\begin{rem}
 The isotopies in both theorems fix $\tb\wh W$ pointwise.
 \end{rem}

\begin{rem}

The final part of the proof of both theorems involves showing that if two admissible foliations of $\wh W$ are defined by the same  foliated ray then they are ambiently isotopic.  We do this in a separate paper~\cite{cc:isorel}.

\end{rem}

\section{Semiproper leaves}

In this section we prove a technical result mentioned in Section~\ref{restrict}. Here, we assume that the foliation is smooth of class $\Ci$  (although $\CII$ or $C^{1+\lip}$ is sufficient), satisfying the requirement that all compact leaves have negative Euler characteristic, and prove the following.  

\begin{prop}\label{nosimple}
If $\dim M=3$ and $L$ is a semiproper leaf, then $L$ has no simple end.
\end{prop}

\begin{proof}
 Semiproper leaves lie at finite levels~\cite[Lemma~8.3.23]{condel1}. We proceed by induction on the level of $L$.  If the level is 0, then $L$ semiproper implies either that $L$ is  compact or lies in an exceptional minimal set.  In the first case, the assertion is true by default and, in the second, it is true by the theorem of Duminy that the endset of $L$ is a Cantor set (unpublished by Duminy, but cf.~\cite{cc:dum}). Assume the assertion true for all semiproper leaves at levels $\le k-1$, $k\ge1$, and let $L$ be semiproper at level $k$.  Then either $L=X$ is proper, hence a local minimal set, or $L\subset X$ where $X$ is an exceptional local minimal set. In either case, there is an open, connected, saturated set $U\subset M$ such that $X$ is a minimal set of $\FF|U$.  Let $e$ be an end of $L$ which is not asymptotic to $L$.  By Duminy's theorem (for exceptional local minimal sets), a simple end must have this property.  Such an end must be asymptotic to a semiproper leaf $F$ at level $\le k-1$.  
 
 First suppose that $F\subset\delta U$ is a border leaf and that $e$ is asymptotic to $F$ from within $\wh U$.  Let $J\subset\wh U$ be an $\FF$-transverse arc with an endpoint $x\in F$ and let $V\subset L$ be an arbitrary neighborhood of $e$.  Then $V\cap J$ must be a sequence $\{x_{i}\}_{i=0}^{\infty}$ converging exactly to $x$.  By the Dippolito semi-stability theorem~\cite{dip}, \cite[Theorem~5.3.4]{condel1}, some loop $\sigma$ on $F$, based at $x$, has holonomy $h_{\sigma}$ such that $h_{\sigma}(x_{k})=x_{k+q}$, some $k\ge0$ and $q\ge1$.  It follows that $h_{\sigma}$ is a contraction and, by the $\Ci$ hypothesis, corresponds to a compact juncture on $F$~\cite[Theorem~8.1.26]{condel1}.  This reference assumes that $F$ is orientable, but cf.~\cite[Section~2.3]{cc:hm} for a way around this difficulty.  Thus, $V$ contains an ``infinite repetition'' of $F$.  Since $F$ has level $\le k-1$ and, if $k=1$, has negative Euler characteristic, $V$ must contain a neighborhood of a nonsimple end.  Since $V$ is an arbitrary neighborhood of $e$, the end $e$ is not simple.
 
 Alternatively, the arbitrary neighborhood $V\subset L$ of $e$ does not accumulate on a border leaf of $U$ from within $U$.  Since $e$ is not asymptotic to $L$, $V$ cannot be contained in a compact part of $\wh U$, hence must meet an arm $A$ of an ``octopus decomposition'' of $\wh U$~\cite[Definition~5.2.13]{condel1}.  Let $B\subset F\subset\tb\wh U$ be a submanifold such that $A=B\x I$ and let $x\in B$ such that $V$ meets $\{x\}\x I$ in at least one point $y$.  Then our hypotheses imply that every loop $\sigma$ in $B$ based at $x$ has total holonomy $h_{\sigma}:\{x\}\x I\to\{x\}\x I$ fixing $y$.  Thus $V$ contains a submanifold homeomorphic to $B$, hence it contains a neighborhood of a nonsimple end.  Since $V$ is arbitrary, $e$ is not a simple end.
\end{proof}

\section{Endperiodic monodromy}\label{endpmono}

 Here, we work in $n$-manifolds, $n\ge3$, without the restrictions on the topology of leaves of Section~\ref{restrict}. Let $\LL$ be a smooth 1-dimensional foliation of $M$ that is transverse to $\FF$ and is
 tangent to $\trb M$, is oriented by the transverse orientation of $\FF$, and fibers by intervals all  components of $\trb M$ that have boundary. (This, of course, restricts the topology of $\FF$. In 3-manifolds we have assumed that the induced foliation of $\trb M$ by $\FF$ has no Reeb components.) Let $W$ be an open, saturated set without holonomy. which is fibered by $\FF$ over $\SI$ and let $\wh\FF$ and $\wh\LL$ denote the foliations induced on $\wh W$ by $\FF$ and $\LL$, respectively.
  
 There is a  sublamination $\XX\subset\LL|\wh W$ consisting of those leaves that do not meet $\tb \wh W$.  This sublamination is a compact sublamination of $\LL|W$.  Indeed, $\XX$ is  closed in $W$ since every leaf of $\wh\LL$ that gets sufficiently close  to $\tb\wh W$ meets $\tb\wh W$.  Since the leaves of $\wh\LL$ in the arms of an octopus decomposition are properly imbedded, compact arcs, $\XX$ is confined to the compact nucleus of the decomposition, hence is compact.  We call $\XX$ the \emph{core lamination} of $\LL|W$.     
 
 We will assume that $\wh W$ is not a foliated product, hence the core lamination
is nonempty.  Indeed, if some leaf $\ell$ of $\LL$ issues from  one
component of $\tb \wh W$ but never reaches another component, the asymptote of
$\ell$ in $\wh W$ will be a nonempty subset of $\XX$.  Thus, if $\XX=\0$,
every leaf of $\LL$ issues from  one component of $\tb \wh W$ and terminates at another component, implying that $\wh W\cong F\times I$, where $F$ is a leaf of $\FF$.   In this case, we will see that Theorem~\ref{cone} is rather trivial, hence we will assume that $\XX\ne\0$. 

The foliation $\LL|W$ can be parametrized as a flow $\Phi_{t}$ preserving $\FF|W$ such that $\Phi_{1}$ sends each leaf to itself.  Indeed, $\FF|W$ defines a fibration 
 $$\pi:W\to\SI$$ and one can lift the  locally defined parameter $ t$ of $\SI=[0,1]/\{0\sim1\}$ to $\LL$. Equivalently, lifting Lebesgue measure on $\SI$ gives a smooth, transverse, holonomy invariant measure for $\FF$. Remark that the parameter becomes unbounded near $\tb\wh W$ and that $\LL\pitchfork \tb\wh W$, hence our flow extends to a flow on $\wh W$ that fixes $\tb\wh W$ pointwise.  If $L$ is a leaf of $\FF|W$, the flow induces a first return map, $\Phi_{1}=f:L\to L$, called the monodromy diffeomorphism of $L$.  If $\FF$ is smooth, the monodromy is endperiodic~(\cite[Definition~2.4 and Theorem~13.16]{cc:hm} which extend to the higher dimensional cases).    In the case that $\FF$ is of class $\COO$, however, the restriction $\FF|K$ to the nucleus of an octopus decomposition  has the property that the monodromy of any noncompact leaf is endperiodic and this will be sufficient for our purposes in that case.
 
 Let $\LL_{f}$ denote $\LL|W$ and  $\wh\LL_{f}$ the induced foliation on $\wh W$, recording thereby the monodromy $f$ that the flow induces on $L$.  We will also denote $\XX\subset\LL$ by $\XX_{f}$. Remark that, while $\LL_{f}$ determines the monodromy $f$, it is not uniquely determined by $f$.

 \section{The Schwartzmann-Sullivan asymptotic cycles}\label{SchSul}
 
 Save mention to the contrary, we continue to work in $n$-manifolds, $n\ge3$, without the restrictions on the topology of leaves of Section~\ref{restrict}.  
 For compact, 1-dimensional laminations such as $\XX_{f}\subset\LL_{f}$,  Sullivan's theory of foliation cycles~\cite{sull:cycles} specializes to the Schwartzmann theory of asymptotic cycles~\cite{sch_cycles}.   We will  take as our main reference the exposition of Sullivan's theory in~\cite[Chapter~10]{condel1}, but taking into account the full theory of de~Rham currents in~\cite[Chapitre~III]{deRham} on noncompact manifolds such as $\wh W$.  As developed in~\cite{condel1}, this theory was concerned primarily with compact foliated manifolds having all leaves without boundary, but as noted there, everything goes through for compact laminations of possibly noncompact manifolds, provided that, again, all leaves have empty boundary.

\subsection{Forms and currents}
Here we summarize the de~Rham theory of homology defined by \emph{currents}, which are related to differential forms exactly as Schwartzian distributions are related to smooth functions.  Since this material is standard, but may be considered  esoteric by some, we give references rather than proofs.

 Slightly modifying the notation of~\cite{deRham} so as to keep track of the degrees of forms and currents, we set $\DD_{p}=\DD_{p}(\wh{W})$, the locally convex topological vector space of \emph{compactly supported} $p$-forms of class $\Ci$.  That is, the underlying vector space is $A^{p}_{c}(\wh{W})$ and the topology $\T$ is generated by the increasing union of the topologies $\T_{k}$ defined by the $C^{k}$ norm $\|\cdot\|_{k}$ on compact sets, $0\le k<\infty$. This norm arises from a choice of $\Ci$ atlas on $\wh W$, and  basic open neighborhoods of 0 in $\T_{k}$ are of the form $V(C,\epsilon)=\{\phi\mid\supp\phi\subset C\text{ and }\|\phi\|_{k}<\epsilon\}$, where $C\subset\wh{W}$ is compact and $\epsilon>0$. The topology $\T$ is independent of the choices. Following de~Rham~\cite[p.~44]{deRham}, we say that a subset $B\subset\DD_{p}$ is bounded if all of its elements have support in a common compact set and $B$ is bounded in the $C^{k}$ norm, $0\le k<\infty$. 
 
 We set $\DD'_{p}$ equal to the strong dual of $\DD_{p}$, the space of continuous linear functionals on $\DD_{p}$.  This is the space of $p$-\emph{currents} on $\wh W$.  The topology on $\DD'_{p}$ can now be defined exactly as in~\cite[Definition~10.1.17]{condel1}.  This makes $\DD'_{p}$ into a locally convex, topological vector space.  It has a notion of bounded subset~\cite[Definition~10.1.20]{condel1}.
Both $\DD_{p}$ and $\DD'_{p}$ are strong duals of one another~\cite[p.~89, Th\'eor\`eme~13]{deRham}.

 De~Rham also introduces a locally convex topological vector space $\EE_{p}$ with underlying vector space $A^{p}(\wh{W})$.  The topology uses a notion of boundedness which is local boundedness in the $C^{k}$ norms, enabling one to define when a sequence of forms converges to $0$.  
The support of a $p$-current  $T$  is the smallest closed subset $C\subset\wh W$ such that $T(\alpha)=0$ whenever $\supp(\alpha)\cap C=\0$.  The  subset $\EE_{p}'\subset\DD'_{p}$ of currents with compact support is exactly the strong dual of $\EE_{p}$ and vice-versa~\cite[p.~89, Th\'eor\`eme~13]{deRham}.  
Using the notion of bounded set in $\EE_{p}$, one  mimics Definition~10.1.17 of~\cite{condel1} to define the locally convex topology on $\EE'_{p}$.  

All of these spaces are Montel, meaning that every bounded subset is precompact.   For the case $p=0$, this is proven in~\cite[p.~70, Th\'eor\`em~VII and p.~74, Th\'eor\`eme~XII]{schwartz}, the general case being similar.

 The exterior derivative $d:\DD_{p}\to \DD_{p+1}$ is continuous, hence has continuous adjoint $\bd:\DD_{p}'\to\DD_{p-1}'$.  Similarly, one has that $d:\EE_{p}\to\EE_{p+1}$ is continuous with continuous adjoint  $\bd:\EE'_{p}\to\EE'_{p-1}$. Since $d^{2}=0$, we see that $\bd^{2}=0$.  The kernel $\ZZ_{p}\subset\EE'_{p}$ of $\bd$ is the space of $p$-cycles and the image $\BB_{p}=\bd(\EE'_{p+1})\subset\EE'_{p}$ is the space of $p$-boundaries. These are closed subspaces of $\EE'_{p}$. The space $H_{p}(\wh{W})=\ZZ_{p}/\BB_{p}$ is the de~Rham  homology of $\wh{W}$, canonically isomorphic to the singular homology.  The homology of the complex
 $(\DD'_{*},\bd)$  gives the dual space to $H^{p}_{c}(\wh{W})$, for each $p\ge0$, which we may think of as homology computed with  locally finite, but possibly infinite, chains and denote it by $H_{p}^{\infty}(\wh{W})=\ZZ^{\infty}_{p}/\BB^{\infty}_{p}$, where, of course, $\ZZ^{\infty}_{p}$ is the kernel of $\bd$ in $\DD_{p}'$ and $\BB_{p}^{\infty}$ its image.

 \subsection{Topologies on homologies and cohomologies.}\label{tops}
 In homology, we topologize $H_{p}(\wh{W})=\ZZ_{p}/\BB_{p}$  and $H_{p}^{\infty}(\wh{W})=\ZZ^{\infty}_{p}/\BB^{\infty}_{p}$ with the quotient topology.
 
 The manifold $\wh W$, if noncompact, is an increasing union 
 $$K_{0}\subset K_{1}\subset\cdots\subset K_{i}\subset\cdots=\wh W,
 $$ 
 where each $K_{i}$ is a compact manifold which is the nucleus of an octopus decomposition of $\wh W$.  This can be used to define natural topologies on $H^{p}(\wh{W})$ and $H^{p}_{\kappa}(\wh{W})$.

  We  topologize $H^{p}(\wh{W})=\ula{\lim}\,H^{p}(K_{i})$~\cite[Proposition~3F.5]{hatch} with the inverse limit topology.  This is the topology induced from the Tychonov topology by the inclusion 
  $$
  \ula{\lim}\,H^{p}(K_{i})\subset H^{p}(K_{0})\x H^{p}(K_{1})\x\cdots\x H^{p}(K_{i})\x\cdots 
  $$ 
  In fact, this is also the quotient topology induced by the complex $(\EE_{*},d)$ (exercise).  In any event, it is clear that the elements of $H^{p}(\wh{W})$, viewed as linear functionals on $H_{p}(\wh{W})$, are continuous and so, this being the vector space dual, it is also the strong dual of $H_{p}(\wh{W})$.  Similarly, the elements of $H_{p}(\wh{W})$, viewed as linear functionals on $H^{p}(\wh{W})$, are continuous (exercise), so homology is a subspace of the strong dual of cohomology. 
 
 Set $K_{i}^{\circ}$ equal to the relative interior of $K_{i}$ in $\wh W$, noting that the inclusions $H^{p}_{\kappa}(K_{i}^{\circ})\hra H^{p}_{\kappa}(K_{i+1}^{\circ})$ 
 induce a natural identification $H^{p}_{\kappa}(\wh{W})=\ura{\lim}\,H^{p}_{\kappa}(K_{i}^{\circ})$.  We give this space the direct limit (weak) topology.  Equivalently, the topology on  $H^{p}_{\kappa}(\wh{W})\subset H^{p}(\wh{W})$ is the induced topology (exercise).

 \subsection{The asymptotic currents for $\XX$.}  While the results in this section are intended for $\XX_{g}$, where $g$ is the monodromy of the leaves in $W$ of a foliation $\wh\HH$ of relative depth one,  we will be more general for possible future applications.  Thus, let $\LL$ be a one dimensional foliation of $\wh W$ transverse to $\tb\wh W$ and of class at least $\COO$ (no foliation $\HH$ in sight).  We assume that, in the arms of a suitable octopus decomposition, the leaves of $\LL$ are the fibers of an interval bundle.  We let $\XX$ be the lamination consisting of the leaves of $\LL$ that do not meet $\tb\wh W$.  Since $\XX$ is closed in $W$ and lies in the compact nucleus of an octopus decomposition, it is compact.  Evidently, $\LL_{g}$  and the core lamination $\XX_{g}$ satisfy all of this.

 The Dirac currents for $\XX$ are the positively oriented, nontrivial tangent vectors to $\XX$. These currents clearly have compact support. The closure in $\DD_{1}'$ of the union of all positive linear combinations of Dirac currents is a closed, convex cone $\CC_{\XX}\subset\DD'_{1}$, called the cone of \emph{asymptotic currents}.  This cone lies on one side of a hyperplane $H= \omega^{-1}(0)$, where $ \omega:\DD'_{1}\to\R$ is a compactly supported 1-form. Indeed, a nonsingular velocity field $v$ of $\LL$, slightly perturbed to a smooth field $w$  and damped off to $0$ outside  a compact neighborhood of $\XX$, together with a Riemannian metric, defines a 1-form $\omega=\<w,\cdot\>$ such that $\omega(v)>0$ on $\XX$.  Since $\XX$ is compact, the asymptotic currents are compactly supported and we can view $\CC_{\XX}\subset\EE'_{1}$.  When working in $\DD'_{1}$, the continuous linear functionals are the compactly supported 1-forms, but when working in $\EE'_{1}$, they are all 1-forms.

 The base $\wh{\CC}_{\XX}=\CC_{\XX}\cap \omega^{-1}(1)$ of the cone $\CC_{\XX}$ is compact (both in $\EE'_{1}$ and in $\DD'_{1}$) by~\cite[Lemma~10.2.3]{condel1}. Once again the proof goes through by the compactness of $\XX$ and the fact that our spaces are Montel. Those continuous linear functionals $ \eta:\DD_{1}'\to\R$ which are strictly positive on $\wh{\CC}_{\XX}$ are exactly the smooth, compactly supported 1-forms on $\wh{W}$ which are transverse to $\XX$ (meaning that they take a positive value on each Dirac current). Similarly, the continuous linear functionals $\eta:\EE_{1}'\to\R$, strictly positive on $\wh{\CC}_{\XX}$, are the smooth 1-forms on $\wh W$ which are transverse to $\XX$. Sullivan applies the Hahn-Banach  theorem, using compactness of the base, to produce interesting 1-forms  that are transverse to $\XX$ (see~\cite[Subsection~10.2]{condel1}).

\subsection{Cones defined by the asymptotic cycles} The cone $\CC_{\XX}\cap \ZZ_{1}$ of \emph{a\-symptotic cycles} is also a closed, convex cone with compact base.  The natural continuous linear surjection $\ZZ_{1}\to H_{1}(\wh{W})$ carries the cone of asymptotic cycles  onto a  convex cone $\C'_{\XX}\subset H_{1}(\wh{W})$ with compact base. The elements of this cone are called \emph{asymptotic classes}.  Compactness of the base implies that this cone  is closed.  There are dual  closed, convex cones in $\hk(\wh{W})$ and $\hI(\wh{W})$:
 \begin{align}\label{C2conedef}
 \C_{\XX}^{\kappa} &= \{[ \eta]\in\hk(\wh{W})\mid [ \eta]([z])\ge0, \forall[z]\in \C'_{\XX}\},\\
  \C_{\XX} &= \{[ \eta]\in\hI(\wh{W})\mid [ \eta]([z])\ge0, \forall[z]\in \C'_{\XX}\}.
\end{align} 
Generally, these do not have compact base.  Indeed, there are interesting cases in which $\C'_{\XX}$ reduces to a single ray, hence the dual cone will be a full half-space.

 Examples of asymptotic cycles are nonnegative, transverse, holonomy invariant measures $\mu$ on $\XX$ that are finite on (transverse) compact sets.  By Sullivan~\cite[Theorem~10.2.12]{condel1}, these are the only ones.  
 
 \begin{theorem}\label{measures}
The asymptotic cycles for $\XX$ are exactly the nonnegative, transverse, holonomy invariant measures on $\XX$ that are finite on compact sets.
\end{theorem}
 
 By a well known theorem of J.~F.~Plante~\cite{plante:meas} and the fact that the leaves of $\XX$, being 1-dimensional, are compact or have linear growth, we obtain the following.
 
 \begin{lemma}\label{nontriv:cycles}
 
 There are nontrivial asymptotic cycles for $\XX$.
\end{lemma}

For the case that $\wh\HH$ is of relative depth~$1$ with monodromy $g:L\to L$,  $\XX=\XX_{g}$ is smooth and we have the following.

\begin{lemma}\label{transverseform}
There is a closed $1$-form $\eta$ on $W$, transverse to $\XX_{g}$, hence no nontrivial asymptotic cycle bounds in $(\EE'_{*},\bd)$.  If $\HH$ is trivial in the arms of some octopus decomposition, this form can be chosen to be compactly supported and no nontrivial asymptotic cycle bounds either in $(\DD'_{*},\bd)$ or $(\EE'_{*},\bd)$.
\end{lemma}

\begin{proof}
Whether the foliation is smooth or  of class $\COO$, $\FF|W$, as a fiber bundle over $\SI$ with smooth leaves,  has a smooth structure. Let $ \omega$ be a closed, nonsingular 1-form defining $\FF|W$. This is clearly transverse to $\XX_{g}$ and $\omega$ takes positive values on all nontrivial asymptotic cycles which, therefore, cannot bound in $(\EE'_{*},\bd)$.  Relative to a suitable octopus decomposition of $\wh W$,  $\XX_{g}$ is contained in the interior of the compact nucleus $K$ of $\wh W$. If $K$ is chosen large enough, then $\HH$  is trivial in $\wh {W}\sm K$. Thus $ \omega= d\gamma$ outside of $K$.  Extending and damping $ \gamma$ off smoothly to be 0 in a neighborhood of $\XX_{g}$, we see that the compactly supported form $\eta= \omega-d \gamma$ is as desired. Again $\eta$ takes positive values on all nontrivial asymptotic cycles, hence these cannot bound either in $(\DD'_{*},\bd)$ or $(\EE'_{*},\bd)$.
\end{proof}

In our more general case,  nontrivial asymptotic cycles may bound.  When they do not, we have the following key result,  the proof of which is an adaptation of the proof of~\cite[Lemma~10.2.8]{condel1}, which assumed the foliated manifold (here, $\wh W$) was compact.  The other hypothesis of that lemma is given by Lemma~\ref{nontriv:cycles}. 
 
 \begin{theorem}\label{intr:cone}
If no nontrivial asymptotic cycle bounds in $(\EE'_{*},\bd)$, $\intr\C_{\XX}\ne\0$ and consists of exactly those classes $[ \eta]\in\hI(\wh{W})$ that are represented by closed 
$1$-forms $\eta$ transverse to $\XX$.  If no nontrivial asymptotic cycle bounds in $(\DD'_{*},\bd)$, the completely analogous assertion holds for $\Ck_{\XX}$ with $\eta$ compactly supported.
\end{theorem}

\begin{proof}
Let $\C_{\XX}^{\circ}\subset\C_{\XX}$ be the set of classes that take strictly positive values on the nonzero elements of $\C'_{\XX}$.  This is nonempty by an application of the Hahn-Banach theorem~\cite[Lemma~10.2.7]{condel1}.  If 
$$
[\eta]=([\eta_{0}],[\eta_{1}],\dots,[\eta_{i}],\dots)\in\ula{\lim}\,\hI(K_{i})=\hI(\wh{W}) 
$$
is in $\C^{\circ}_{\XX}$, then $\eta_{0}$ takes positive values on all nontrivial asymptotic cycles.  These cycles are all supported in the compact manifold $K_{0}$, hence the analogous result for $\wh{W}=K_{0}$ is proven in~\cite[Lemma~10.2.8]{condel1}. (Compactness is used to guarantee that $\hI(K_{0})$ is finite dimensional.) Thus $[\eta_{0}]\in U\subset\hI(K_{0})$, where $U$ is the interior of the dual cone there.  Then  $$\UU=U\x\hI(K_{1})\x\hI(K_{2})\x\cdots\x\hI(K_{i})\x\cdots$$ is an open set in the Tychonov topology and $\UU\cap\hI(\wh{W})=\C^{\circ}_{\XX}$, hence this set is open. We prove that it is the entire interior of $\C_{\XX}$ by showing that every $[\alpha]\in\C_{\XX}$ that vanishes on some nontrivial asymptotic cycle $T$ is in the frontier.  Indeed, if $[\eta]\in\C^{\circ}_{\XX}$, form the line of classes $t[\eta]+(1-t)[\alpha]$, $-1\le t\le 1$. For $0<t\le1$, these classes are in $\C^{\circ}_{\XX}$, but for $-1\le t<0$, they take negative values on $T$, hence are not in $\C_{\XX}$.

It remains to show that if $[\eta]\in\C^{\circ}_{\XX}$, then $\eta$ is cohomologous to a closed form $\eta'$ that takes positive values on the entire base $\wh{\CC}_{\XX}$ of the cone of asymptotic currents.  The kernel of $[\eta]:H_{1}(\wh{W})\to\R$ is a closed hyperplane in that vector space and its pre-image  $V\subset\ZZ_{1}$ is the kernel of $\eta:\ZZ_{1}\to\R$,  a closed vector subspace containing $\BB_{1}$ and meeting $\CC_{\XX}\cap\ZZ_{1}$ only at the vertex of that cone, hence meeting $\CC_{\XX}$ itself only at the vertex.  Using the standard  Hahn-Banach argument (cf.~\cite[Section~10.2]{condel1}), we find a continuous linear functional $\eta':\EE'\to\R$ which is strictly positive on $\wh{\CC}_{\XX}$ and vanishes identically on $V$.  Since $\BB_{1}\subset V$, $\eta'$ is a closed form transverse to $\XX$.  Since $[\eta']$ vanishes on the kernel of $[\eta]$ and is not trivial, $[\eta]$ and $[\eta']$ are positive multiples of each other, completing the proof for the cone $\C_{\XX}$.

For the cone $\Ck_{\XX}$, the proof is quite similar, but there are sufficient differences that we will give details. The fact that $\C^{\kappa\circ}_{\XX}\ne\0$ is guaranteed by the standard Hahn-Banach argument.  The fact that this is the interior of $\Ck_{\XX}$ follows from the corresponding fact for $\C_{\XX}$ via the relative topology.  If $[\eta]\in\intr\Ck_{\XX}$, we take $\eta$ compactly supported,  hence it is a continuous linear functional $\eta:\DD'_{1}\to\R$.  Restricting to the space $\ZZ_{1}^{\infty}$ of closed currents, we obtain a closed subspace $V=\eta^{-1}\{0\}$ which contains the space $\BB_{1}^{\infty}$ of boundaries.  Viewing $\CC_{\XX}$ as a cone in $\DD'_{1}$, we again apply the Hahn-Banach theorem to produce a compactly supported form $\eta'$ strictly positive on $\wh{\CC}_{\XX}$ and vanishing on $V$. Thus $\eta'$ is  compactly supported and transverse to $\XX$.  Since $\BB_{1}^{\infty}\subset V$, $\eta'$ is closed.  Via the continuous injection $\EE'_{1}\hra\DD'_{1}$, $V$ pulls back to a space having as image in $H_{1}(\wh{W})$  the kernel of $[\eta]:H_{1}(\wh{W})\to\R$.  Since $[\eta']$ also vanishes on this space and is nontrivial,  $[\eta]$ and $[\eta']$ are positive multiples of each other.
\end{proof}

Varying the  monodromy $g$ within its  isotopy class will give different  cones $\C_{g}=\C_{\XX_{g}}$ and $\Ck_{g}=\Ck_{\XX_{g}}$.  In Sections~\ref{HMcones} and~\ref{HMcones1}, when $\dim M=3$, we use the Handel-Miller theory~\cite{cc:hm} to find a monodromy  $h$ which gives  maximal cones $\C_{h}$ and $\Ck_{h}$ among all cones given by monodramy maps in the  isotopy class of $g$.  The monodromy $h$ has the ``tightest'' dynamics in its isotopy class, and determines the minimal dual homology cones. 

There is a serious smoothness problem with this monodromy  $h$ given by Handel-Miller theory~\cite{cc:hm}.  But $h$ is not unique in its isotopy class and a suitable choice will give a strongly transverse $\LL_{h}$ which is smooth except at finitely many closed leaves in $\XX_{h}$ where it is only of class $\COO$.  Since the Schwartzmann-Sullivan theory works when $\XX$ is only integral to a $\CO$ line field, all will be well.

The notations $\C_{g}$, $\Ck_{g}$ and $\C'_{g}$, suggest dependence only on $g$, rather than on the choice of $\XX_{g}$.  In Section~\ref{indchoice}, we will see that this is correct when $\dim M=3$.  At the level of currents, the notation $\CC_{g}$ is a bit of an abuse.
 
 \subsection{Homology directions}

It will be important to characterize a particularly simple spanning
set of $\C_{\XX}'$, the so called ``homology directions'' of Fried~\cite[p.~260]{fried}.  Assuming that $\LL$ has been parametrized as a
nonsingular $\COO$ flow $\Varphi _{t}$, select a point $x\in\XX$ and
let $\Gamma $ denote the $\Varphi $-orbit of $x$.  If this is a
closed orbit, it defines an asymptotic cycle which we will denote by
$\ol{\Gamma }$.  If it is not a closed orbit, let $\Gamma _{\tau
}=\{\Varphi _{t} (x)\mid 0\le t\le \tau \}$.  Let $\tau
_{k}\uparrow\infty$ and set $\Gamma _{k}=\Gamma _{\tau _{k}}$.  After
passing to a subsequence, we obtain an asymptotic current $$ \ol{\Gamma}
= \lim_{k\ra\infty}\frac{1}{\tau _{k}}\int_{\Gamma _{k}}^{}. $$  In fact this is a \emph{cycle}.   One calls $\Gamma_{k}$ a ``long, almost closed orbit'' of $\XX_{h}$.  Its endpoints lie in the compact set $\XX$  and it can be closed by adding a uniformly bounded curve in $M$. These are averaged out in the limit and the corresponding singular cycles, also called ``long, almost closed orbits'', are denoted by $ \Gamma'_{k}$.  The cycles $(1/\tau_{k}) \Gamma'_{k}$ also limit on $\ol{\Gamma}$, proving that it is a cycle.

\begin{lemma}\label{long_aco}
The asymptotic current $\ol{\Gamma }$, obtained as above, is an asymptotic
cycle.
\end{lemma}

Another proof can be given by appealing to Stokes's theorem.

\begin{defn}
All  asymptotic cycles $\ol{\Gamma} $, obtained as
above, and  their homology classes are called \emph{homology directions} of
$\XX$.
\end{defn}

By abuse, we will denote both the cycle and its homology class by $\ol \Gamma$.

An elementary application of ergodic theory proves the following
(see~\cite[Proposition II.25]{sull:cycles} and~\cite[Proposition~10.3.11]{condel1}).

\begin{lemma}\label{span}
Any asymptotic cycle $\mu$ can be arbitrarily well approximated by
finite, nonnegative linear combinations $\sum_{i=1}^{r}a_{i}\ol{\Gamma }_{i}$ of homology
directions. If $\mu \ne0$, the coefficients $a_{i}$ are strictly positive and
their sum is bounded below by a constant $b_{\mu }>0$ depending only on $\mu
$.
\end{lemma}

\subsection{The independence of the cones from various choices}\label{indchoice}

In this subsection, we restrict to the case that $\dim M=3$, imposing the usual restrictions on the topology of the leaves of $\FF$, namely that the compact leaves have strictly negative Euler characteristic and no leaf is contractible nor has the homotopy type of the circle.  We assume that $\wh\FF$ is of relative depth~$1$ on $\wh W$ and  consider $\LL_{f}$ and $\XX_{f}$, corresponding to the  monodromy homeomorphism $f$ on a leaf $L$ of $\FF|W$.  We will not need $f$ nor the foliations to be smooth.

Let us first note that long, almost closed orbits and homology directions, as classes in the singular homology $H_{1}(\wh{W})$,  are clearly well defined even when $\LL_{f}$ and $\XX_{f}$ are only $\CO$.    While most of the theory of asymptotic classes fails when these objects are not at least integral to a $\CO$ line field, we can still define the cone $\C'_{f}$ as the closure in $H_{1}(\wh{W})$ of the set of nonnegative linear combinations of homology directions.  This is a closed, convex cone which, when $\XX_{f}$ is of class $\COO$, coincides with the cone already defined (Lemma~\ref{span}).  We are going to show that  this cone depends only on $f$, not on the choice of $\LL_{f}$.

 Let $\LL$ and $\LL_{\sharp}$ be strongly transverse to $\wh\FF$.   
 Let $\XX$ and $\XX_{\sharp}$ be the respective core
laminations and let $\C'_{\XX}$ and $\C'_{\XX_{\sharp}}$ denote the corresponding cones in $H_{1}(\wh{W})$, $\C_{\XX}$ and $\C_{\XX_{\sharp}}$ those in $\hI(\wh{W})$, and $\Ck_{\XX}$ and  $\Ck_{\XX_{\sharp}}$ those in $\hk(\wh{W})$.  

In~\cite{cc:cone}, the following elementary theorem was deduced as a corollary of a much deeper result (Lemma~4.10 in that reference) which we attempted to deduce from results of M.~E.~Hamstrom~\cite{ham:disk-holes, ham:torus, ham:homeo}.  A correct proof of that lemma  needs a deep result of T.~Yagasaki~\cite{yag}, but we omit this because we do not need it.   

 \begin{theorem}\label{samecones}
Let  $L$ be a  leaf of $\FF|W$, and let $\LL$
and $\LL_{\sharp}$ be  $1$-dimensional foliations of $\wh W$,  strongly transverse to $\wh\FF$,  having respective core laminations $\XX$ and $\XX_{\sharp}$, and inducing the same endperiodic monodromy  $f:L\to L$.    Then $\C'_{\XX}=\C'_{\XX_{\sharp}}$, $\C_{\XX}=\C_{\XX_{\sharp}}$ and $\Ck_{\XX}=\Ck_{\XX_{\sharp}}$.
\end{theorem}

We will show that the homology directions determined by the long, almost closed
orbits of $\XX$ are exactly the same as the ones for $\XX_{\sharp}$ and the theorem will follow.  Our arguments are carried out in the $\CO$ category for an arbitrary connected surface $L$ that is either compact with strictly negative Euler characteristic or is noncompact, noncontractible and not homotopy equivalent to the circle. By our basic assumptions, the leaves of $\FF$ satisfy this condition.

  Let $I$ be the compact interval
$[0,1]$ and consider the product $L\times I$. (One obtains such a product, for instance,  by
cutting $W$ apart along $L$ and taking as the interval fibers the
resulting segments of the leaves of $\LL$.) For each $x\in L$, denote by
$I_{x}$ the interval fiber with endpoints $\{x\}\times\{0,1\}$. Consider a
second fibration of $L\times I$ by intervals $J_{x}$, requiring that the
endpoints of $J_{x}$ coincide with those of $I_{x}$, for all $x\in L$.  (By
the hypothesis on $\LL_{\sharp}$, this second fibration arises in our case by cutting
apart along $L$ and using the segments of leaves of $\LL_{\sharp}$ as fibers.)
For each $x\in L$, let $\alpha_{x}$ denote the loop in $L\times I$ obtained
by following $I_{x}$ from $(x,0)$ to $(x,1)$ and then following $J_{x}$ from
$(x,1)$ to $(x,0)$.  Finally, if $p:L\times I\to L$ is the canonical
projection, let $\beta_{x}=p\circ\alpha_{x}$, a loop in $L$.

\begin{lemma}\label{commutes}
Let $x_{0}\in L$ and set $\delta=\beta_{x_{0}}$.  If $\gamma(s)$, $0\le s\le 1$, is any
other closed curve in $L$ based at $x_{0}$, then $\gamma\cdot\delta =
\delta\cdot\gamma$ in $\pi_{1}(L,x_{0})$.
\end{lemma}

\begin{proof}
Define $F(s,t) = \beta_{\gamma(s)}(t)$. Then $F(s,0) = \gamma(s) =
F(s,1)$. Also $F(0,t) = F(1,t) = \beta_{x_{0}}(t) = \delta(t)$. Because of
this last, we can view $F$ as a map from the cylinder $S^1\times [0,1]$ into
$L$. The curve obtained by following $F(0,t)$, $0\le t\le 1$, followed by
$F(s,1)$, $0\le s\le 1$, and then $F(0,1-t)$, $0\le t\le 1$, is the composite
loop $\delta\cdot\gamma\cdot\delta^{-1}$. We show how to deform this curve
continuously to $\gamma$, keeping the basepoint $x_{0}$ fixed throughout the deformation.

Let $\sigma_{t}$ be the curve obtained by following $F(0,\tau)$, 
$0\le\tau\le t$, followed by $F(s,t)$, $0\le s\le 1$, followed by
$F(0,t-\tau)$, $0\le \tau\le t$. Since $\sigma_{0} = \gamma$ and 
$\sigma_{1} = \delta\cdot\gamma\cdot\delta^{-1}$, we have the desired 
deformation. 
\end{proof}

\begin{cor}\label{0homotopic}
If there exists an $x_{0}\in L$ so that $J_{x_{0}}$ cannot be deformed into $I_{x_{0}}$ keeping the endpoints fixed, then $L$ is either compact with nonnegative Euler characteristic, or is contractible, or has the homotopy type of the circle.
\end{cor}

\begin{proof}
The hypothesis implies that $\alpha_{x_{0}}$ is essential in $L\times I$, and so $\beta_{x_{0}}$ is essential in $L$ and thus is a nontrivial element of $\pi_{1}(L,x_{0})$. By Lemma~\ref{commutes}, every element of $\pi_{1}(L,x_{0})$ commutes with $\beta_{x_{0}}$. Thus, if $L$ is compact, $\chi(L)\ge0$. If $L$ is not compact and not contractible, $L$ is homotopically equivalent
to a bouquet $B$ of circles. The only bouquet of circles that contains a
nontrivial element of $\pi_1(B,*)$ that commutes with every other element of
$\pi_1(B,*)$ is one circle.
\end{proof}

\begin{proof}[Proof of \emph{Theorem~\ref{samecones}}]
 Let $\ol{\Gamma }\in H_{1}
(\wh{W};\R) $ be a homology direction for $\XX$.

First assume that $\ol{\Gamma }$ is not represented by a closed orbit in
$\XX$ and write 
 $$
\ol{\Gamma } = \lim_{k\to\infty}\frac{1}{\tau _{k}}[\Gamma' _{k}],
$$ a limit in $H_{1} (M)$ of the homology classes of long, almost closed
orbits. The numbers $\tau _{k}$ are the ``lengths'' of $\Gamma _{k}$
(measured by the transverse, invariant measure for $\FF|W$) and
increase to $\infty$ with $k$.  Thus, except for a uniformly bounded arc in $L$,  $\Gamma '_{k}$ is a
sequence of segments, $\sigma _{1},\dots,\sigma _{n_{k}}$ of an orbit in
$\XX$, each starting and ending in $L$.  There is a corresponding
sequence $\sigma' _{1},\dots,\sigma' _{n_{k}}$ of segments of an orbit in
$\XX_{\sharp}$ such that $\sigma _{i}\text{ and }\sigma '_{i}$ have the same
endpoints and the same lengths, $1\leq i\leq n_{k}$. By Corollary~\ref{0homotopic}, these respective segments are
homotopic by a homotopy that keeps their endpoints fixed.  Thus, we see that $\ol{\Gamma }$ is also a homology direction for $\XX_{\sharp}$.  In the case that
$\ol{\Gamma }$ is represented by a closed orbit, the argument adapts and is
simpler. Finally, the roles of $\XX\text{ and }\XX_{\sharp}$ can be interchanged,
proving that the two laminations have the same homology directions.
\end{proof}

\begin{prop}\label{C0} Let $\LL\text{ and }\LL_{\sharp}$ be two  $1$-dimensional  
foliations strongly transverse to $\FF|\wh{W}$.  Suppose that the respective core
laminations $\XX\text{ and }\XX_{\sharp}$ are $\CO$-isotopic by an isotopy
$\varphi _{t}:\XX\hra M$, $\varphi _{0}=\id_{\XX}$ and $\varphi _{1} (\XX)=\XX_{\sharp}$, such that $\varphi _{t} (x)$
lies in the same leaf of $\FF|W$ as $x$, for $0\le t\le 1$,
$\forall\,x\in \XX$. Then $\C'_{\XX}=\C'_{\XX_{\sharp}}$ and $\C_{\XX}=\C_{\XX_{\sharp}}$.
\end{prop}

\begin{proof}  Parametrize the two foliations as  flows using the
same transverse invariant measure $ \theta$ for $\FF$.  Since $\FF$ is
leafwise invariant under the  isotopy, the flow parameter is
preserved and the long, almost closed orbits of $\XX$ are isotoped to
the long, almost closed orbits of $\XX_{\sharp}$.  Homotopic singular cycles
are homologous and the assertions follow.
\end{proof}

The
property that points of $W$ remain in the same leaf of $\FF|W$
throughout the isotopy will be indicated by saying that
$\FF|W$ is leafwise invariant by $\varphi _{t}$.

\begin{cor}\label{conj}
Let $f:L\ra L$ be the  endperiodic first return homeomorphism induced on a
 leaf $L$ of $\FF|W$ by a strongly transverse  $1$-dimensional
foliation $\LL_{f}$ of class $\CO$.  If $\varphi_{t}:L\ra L$ is an isotopy of $\phi_{0}=\id$ to a homeomorphism 
$\phi=\phi_{1}$, then $\C'_{f}=\C'_{\varphi f\varphi^{-1}}$.
\end{cor}

\begin{proof}
Let $N$ be a closed normal neighborhood of $L$ in $W$ which is a foliated product with leaves the leaves of $\FF$ meeting $N$ and normal fibers the arcs of $\LL\cap N$.  Write $N=L\x[-\epsilon,\epsilon]$ and consider each arc $\ell_{x}$ of a leaf of $\LL$ issuing in the positive direction from $(x,\epsilon)\in L\x\{\epsilon\}$ and first returning to $N$ at $(f(x),-\epsilon)$.  In $N$, replace each arc $\tau_{y}=\{y\}\x[-\epsilon,\epsilon]$ of $\LL\cap N$ with an arc $\sigma_{y}:[-1,1]\to N$ defined by 
\begin{align*}
\sigma_{y}(t) &= (\phi_{t+1}(y),\epsilon t),\,-1\le t\le0,\\
\sigma_{y}(t) &= (\phi_{t}^{-1}(\phi_{1}(y)),\epsilon t),\,0\le t\le1.
\end{align*}  Notice that this still connects $(y,-\epsilon)$ to $(y,\epsilon)$.  We construct an ambient leaf-preserving isotopy $\psi$, supported in $N$ and carrying each $\tau_{y}$ to $\sigma_{y}$, by
\begin{align*}
\psi_{s}(y,\epsilon t)  &= (\phi_{s(t+1)}(y),\epsilon t),\,-1\le t\le0,\,0\le s\le1,\\
\psi_{s}(y,\epsilon t) &= (\phi_{st}^{-1}(\phi_{s}(y)),\epsilon t),\,0\le t\le1,\,0\le s\le1.
\end{align*}
We obtain $\LL'$ from $\LL$ by replacing $\tau_{y}$ with $\sigma_{y}$, $\forall y\in L$, observing that the monodromy induced by $\LL'$ on $L=L\x\{0\}$ is $\phi\circ f\circ\phi^{-1}$.  The assertion follows by Proposition~\ref{C0}.
\end{proof}

\subsection{Foliated forms}\label{fforms}

In this subsection we assume that the smooth $1$-dimensional foliation $\wh\LL$ on $\wh W$ is strongly transverse to the foliation $\wh\FF$ on $W$, relatively almost without holonomy and induced by the foliation $\FF$ of $M$. $\wh\FF$ may be dense leaved or of relative depth~$1$.  We will need that $\wh\LL\sm\XX$ is smooth and that $\LL$ itself is of class $\COO$.  Our discussion will be valid for all dimensions $\ge3$ and without restrictions on the topology of leaves.

\begin{defn}
A closed, nonvanishing form $\eta\in A^{1}(W)$, defining a foliation $\FF$ of $W$, \emph{blows up smoothly} at $\tb\wh W$ if $\eta_{x}$ becomes unbounded as $x\to\tb\wh W$ and $\FF$ completes to a foliation $\wh\FF$ of $\wh W$ by adjunction of the components of $\tb\wh W$ as leaves, $\wh\FF$ being of class $\Ci$ and $\Ci$-flat at $\tb\wh W$.
\end{defn}

\begin{defn}\label{ff}
A 1-form $\eta\in A^{1}(W)$ is a \emph{foliated form} if it is closed, nowhere vanishing and blows up smoothly at $\tb\wh W$.
\end{defn}

We will prove the following.

\begin{theorem}\label{folforms}
 The  open cone $\intr\C_{\XX}$ consists of  classes in $H^{1}(\wh{W})$ that can be represented by foliated forms transverse to $\LL$.  Similarly, the interior of $\Ck_{\XX}$ consists of classes represented by $\LL$-transverse foliated forms $\eta$ which become exact in $W\sm K_{i}$ for large enough values of $i$, defining a product foliation there.  If  $\LL=\LL_{f}$ for the monodromy of a smooth foliation $\FF$ of relative depth one, this product foliation can be chosen to agree with $\FF$ in this region.
 \end{theorem}
 
 Remark that foliated forms only live in $W$, not in $\wh W$, but $\hI(\wh W)=\hI(W)$  and any form representing a class in $W$ can be taken to be equal to that representing the class in $\wh W$ outside of any small neighborhood of $\tb\wh W$.  If the foliated form is exact in the arms of some octopus decomposition, it is clear that it represents a class in $\hk(\wh{W})$.

Thus, we call $\C_{\XX}$ and $\Ck_{\XX}$  foliation cones associated to $\XX$.  If $\XX=\XX_{f}$, where $\wh\FF$ is of relative depth~$1$ with monodromy $f$,  they are foliation cones associated to $f$.  The rays out of the origin meeting the interior of $\C_{\XX}$ or of $\Ck_{\XX}$  correspond to foliations $\wh\HH$ of $\wh W$ that have holonomy only along $\tb\wh W$ ($\Ci$ tangent to the identity), are transverse to $\wh\LL$ and extend $\FF|(M\sm W)$ over $M$ to a foliation of class $\COO$.  Even if $\FF$ was of class $\Ci$ and $\Ci$-flat along the border leaves on $W$, the extension may fail to be smooth unless the foliations $\wh\FF$ and $\wh\HH$ are products in the arms of an octopus decomposition.  The rational rays correspond to foliations defined by forms $\eta$ with period group infinite cyclic, defining foliations of $\wh W$ that fiber $W$ over $\SI$.  The rest of the rays in $\intr\C_{\XX}$ consist of classes  having period group dense in $\R$ and so define foliations that are dense leaved in $W$.  Recall that compactness of junctures in $\CII$ foliations  forces   $\wh\FF$ to be trivial in the arms when $\FF$ is smooth and $\FF|W$ fibers $W$ over $\SI$.

\begin{proof}[Proof of \emph{Theorem~\ref{folforms}}]
  Fix a class $[\eta]\in\intr\C_{\XX}$, the 1-form $\eta\in[\eta]$ being defined on $\wh W$ and transverse to $\XX$ (Theorem~\ref{intr:cone}). If $\FF$ is of class $\Ci$, select this form to be compactly supported in $W$, $[\eta]\in\Ck_{\XX}$. Select a neighborhood $U$ of $\XX$ such that $\eta\tr\LL|U$.  We need to show that $\eta$ is cohomologous to a foliated form. Note that, if $\eta$ is compactly supported in $K_{i}$, the foliated form will  automatically be exact in $W\sm K_{i}$.

Given $x\in W\sm\XX$, 
let $s(t)$ be the smooth trajectory along $\wh\LL$ in  $W\sm\XX$, smoothly reparametrized so that $x=s(0)$ and $s(\pm1)\in F_{\pm}$.  Here, $F_{+}$ is the union of outwardly oriented leaves of $\tb\wh{W}$ and $F_{-}$ the union of inwardly oriented ones. For some choices of $x$ both signs may be possible and for others only one. For definiteness, consider the case $s(-1)\in F_{-}$.  Define a tubular neighborhood $V_{x}=D\x[-1,3/4)$ of $s$ so that $s(t)=(0,t)$ and $\{z\}\x[-1,3/4)$ is an arc in $\wh\LL$, $\forall z\in D$.  Here, $D$ is the open unit $(n-1)$-ball with polar coordinates $(r, \theta_{1},\dots,\theta_{n-2})$, $0\le r<1$.  This gives cylindrical coordinates $(t,r, \theta_{1},\dots,\theta_{n-2})$ on $V_{x}$.  On $V_{x}$, define a smooth, real valued function  $$\ell_{x}(t,r, \theta_{1},\dots,\theta_{n-1})=\ell_{x}(t,r)=\ell(t)\lambda(r),$$ where $\ell(t)=t-1$, $-1\le t\le1/2$, and damps off to 0 smoothly and with positive derivative  as $t\to3/4$, and $ \lambda(r)\equiv1$, $0\le r\le1/2$, and damps off to  0 smoothly  through positive values as $r\to1$.  Thus, $ \ell _{x}(t,r)$ vanishes outside of $V_{x}$ and $d \ell _{x}$  is transverse to $\wh\LL$ in $V_{x}$.   Let $V'_{x}\subset V_{x}$ be the neighborhood of $x$ defined by $-1\le t<1/2$ and $0\le r<1/2$.   Perform an analogous construction for trajectories out of $x$ with $s(1)\in F_{+}$.

Suitable choices of these open cylinders  (using the local compactness) give a locally finite open cover $\{U,V'_{x_{1}},V'_{x_{2}},\dots\}$ of $\wh W$.  For suitable choices of positive constants $c_{i}$, set $ \ell=\sum_{i=1}^{\infty} c_{i}\ell_{x_{i}}$, a smooth function, supported in $W\sm\XX$, with $d\ell\tr\LL$ outside of a compact neighborhood of $\XX$ in $U$.  Since $\eta$ is bounded in any compact region of $\wh W$ and is transverse to $\LL$ in $U$, we can choose the coefficients $c_{i}>0$ large enough  that $\eta'=\eta+ d \ell$ is a closed form in $\wh W$, cohomologous to $\eta$ and transverse to $\wh\LL$.  This form might be badly behaved at $\tb\wh W$, hence we must modify it by adding on a suitable exact form supported in a neighborhood of the boundary leaves.

Let $V=F_{-}\x[0,1)$ be a normal neighborhood of $F_{-}$ in $\wh W$, the fibers being arcs in leaves of $\wh\LL$. 
Let $\lambda$ be a smooth function, supported in the deleted  normal neighborhood $V\sm F_{-}$, depending only on the normal parameter $t$, and having $\lambda'(t)\ge0$,  with $\lambda'(t)=e^{1/t^{2}}$  near $F_{-}$.  Make a similar construction near $F_{+}$.  Now $\wt{\eta}=\eta'+d\lambda$ is everywhere transverse to $\LL$, hence nonsingular on $W$, it is cohomologous to $\eta'$ and it becomes unbounded at $\tb \wh W$.  We must show that $\ker\wt\eta$ extends $\Ci$-smoothly to a plane field on $\wh W$ by adding on the tangent planes to $\tb \wh W$.  For this, set $\ol{\eta}=\wt{\eta}/\lambda'=\eta'/\lambda'+dt$, a form defined  on a small enough deleted neighborhood  of $F_{-}$.  This form is no longer closed but satisfies $\ker\ol{\eta}=\ker\wt\eta$ in that neighborhood.  Since $\eta'$ is bounded on $\wh W$, it is clear that $\ol{\eta}$ approaches $dt $ in the $\Ci$ topology as $t\to0$ and that the resulting foliation of $\wh W$ is of class $\Ci$ which is $\Ci$-trivial at $\tb\wh W$.
 After a similar construction in a normal neighborhood of $F_{+}$, we obtain a foliated form, again denoted by $\wt\eta$, transverse to $\LL$.
 
 Note that, if we choose $\eta$ compactly supported, then $\wt{\eta}$ is exact outside of some compact region  in $\wh W$.  If $\FF$ is smooth, we can choose $\wh\LL$ to be strongly transverse to $\wh\FF$ and assume that $\FF|W$ is defined by a foliated form $\omega$ which is exact outside of a compact region.  In this case,  we will choose $\wt\eta$ to agree with this exact form.  For $i$ large, $\omega|({W}\sm K_{i})=d\gamma$ and $\supp\eta\subset K_{i}^{\circ}$.  Choose $r>0$ so that $\{U,V'_{x_{1}},V'_{x_{2}},\dots,V'_{x_{r}}\}$ covers $K_{i+1}$.  Smoothly damp $\gamma$  off to $0$ in a neighborhood $N$ of $K_{i}$ in $K_{i+1}$ so that it becomes $0$ near $K_{i}$ and is unchanged in $\wh{W}\sm N$,  extending it by $0$ over $K_{i}$. Call this new function $\wt\gamma$.  Similarly, damp $\lambda$ off to $0$ in $V\sm K_{i+1}$ so that $d\lambda$ remains transverse to $\LL$ in $V\sm K_{i+1}$ and becomes $0$ in $V\sm K_{i+2}$, calling this new function $\wt\lambda$. Now, $\wt\eta=\eta+\sum_{j=1}^{r}c_{j}d\lambda_{x_{j}}+d\wt\lambda+d\wt\gamma$ is as desired for suitably large choices of the constants $c_{j}.$  
The proof of Theorem~\ref{folforms} is complete.  
\end{proof}

\section{Foliations of class $C^{2}$}\label{HMcones}

In this section, we prove Theorem~\ref{conesmooth}. We restrict our attention to the case that $\FF$ is smooth,  $\dim M=3$ and all compact leaves have negative Euler characteristic.  Consequently, if $\wh\FF$ is a foliation of $\wh W$ of  relative depth one and $L$ is a leaf of $\FF|W$, the monodromy $f:L\to L$ is endperiodic~\cite[Theorem~13.16]{cc:hm} and $L$ has no simple ends.  

\begin{rem}

Most of~\cite{cc:hm} concerns endperiodic automorphisms of surfaces with finite endset. However,~\cite[Section~13.4]{cc:hm} extends the definitions given for surfaces with finite endsets to surfaces with infinite endsets and with these definitions most of the results of ~\cite{cc:hm} are true. We use the  terminolgy, notation, and results of~\cite{cc:hm} for surfaces with infinite endset without further comment.

\end{rem}

\begin{rem}

Our treatment of Handel-Miller theory in this paper is a brief outline. The reader should consult~\cite{cc:hm} for details.

\end{rem}

Given an endperiodic automorphism $f:L\to L$, Handel-Miller theory~\cite[Section~4]{cc:hm} gives a pair of transverse, geodesic  bilaminations $(\Lambda_{+},\Lambda_{-})$ associated to $f$. In~\cite[Section~10]{cc:hm} we give a set of axioms and show that the pseudo-geodesic bilaminations $(\Lambda_{+},\Lambda_{-})$ associated to $f$ satisfying the axioms  are ambiently isotopic to the Handel-Miller geodesic ones.   In particular, they are unique up to ambient isotopy. The following is~\cite[Definition~10.7]{cc:hm}.

\begin{defn}

The pseudo-geodesic bilaminations $(\Lambda_{+},\Lambda_{-})$ in the previous paragraph are called the \emph{Handel-Miller laminations} associated to $f$,

\end{defn}

The following is~\cite[Definition~12.4]{cc:hm}.

\begin{defn}
Suppose $\wh\FF$ is a foliation of $\wh W$ of  relative depth one and $L$ is a leaf of $\FF|W$ with endperiodic  monodromy $f:L\to L$. If $\wh\LL$ is a $1$-dimensional foliation transverse to $\wh\FF$ inducing
 monodromy $h:L\to L$ preserving the Handel-Miller laminations $(\Lambda_{+},\Lambda_{-})$ associated to $f$, then $h$ is called  the \emph{Handel-Miller monodromy} of $L$. 
\end{defn}

\begin{rem}
Given the laminations $(\Lambda_{+},\Lambda_{-})$ associated to the endperiodic monodromy $f:L\to L$, by~\cite[Theorem~8.1, Corollary~10.11]{cc:hm}, there exists   Handel-Miller monodromy $h:L\to L$. 
\end{rem}

\begin{rem}
By~\cite[Theorem~11.1]{cc:hm} the laminations $(\Lambda_{+},\Lambda_{-})$ can be chosen to be smooth and $h$ a diffeomorphism. We will always do this and take Handel-Miller monodromy to be smooth without further comment. 
\end{rem}

\begin{rem} 
In fact it will be necessary to further tighten $h$ in a bounded open subset of $L$, as we will show shortly.  This will give what we will call the \emph{tight} Handel-Miller monodromy $h$ (Definition~\ref{tight}).
\end{rem}

By the above remarks, there is a smooth $1$-dimensional foliation $\LL_{h}$, transverse to $\FF$ and inducing the tight  Handel-Miller monodromy.  We will show that the  cone $\Ck_{h}$ is the maximal foliation cone corresponding to the isotopy class of $f$.  Note that ``the'' tight Handel-Miller monodromy is a misnomer.  There are infinitely many that are mutually isotopic, but we will show that  $\Ck_{h}$ is independent of the choice.  For this, we need to show that the homology cone $\C'_{h}$ is independent of the choice.

\begin{defn}

The homology cone $\C'_{h}$ is called the \emph{Handel-Miller  cone} where $h$ is tight (and smooth) Handel-Miller monodromy.

\end{defn}

We need to investigate the asymptotic cycles for $\XX_{h}$ more carefully.

\subsection{The invariant set}\label{invset}
The lamination $\XX_{h}$ is the $\LL_{h}$-saturation of the invariant set $X_{h}=L\sm(\UU_{+}\cup\UU_{-})$, the set of points that do not escape to ends of $L$ under forward or backward iteration of $h$.  Set $ X^{*}_{h}=\Lambda_{+}\cap \Lambda_{-}\subseteq X_{h}$, the so called  \emph{meager invariant set}.  Generally, $X_{h}^{*}\ne X_{h}$.  Referring to~\cite{cc:hm} (where $X_{h}$ was denoted by $\II$~\cite[Definition~5.11]{cc:hm} and $X_{h}^{*} $ by $\KK$~\cite[Definition~9.1]{cc:hm} ), we describe the set $X_{h}\sm X^{*}_{h}$.

The complement $L\sm (\Lambda_{+}\cup \Lambda_{-})$ generally has infinitely many components, ``most'' of which lie in the $\pm$-escaping set $\UU_{+}\cup\UU_{-}$, hence are disjoint from $X_{h}$.  Those components that do not lie in the $\pm$-escaping set are the nuclei $N_{i} $ of finitely many principal regions $P_{i}$, $1\le i\le n$.    The nucleus $N_{i}$  is a compact, connected surface meeting $X^{*}_{h}$ in finitely many vertices, its boundary being piecewise smooth, made up alternately of arcs of leaves of $\Lambda_{+}$ and arcs of leaves of $\Lambda_{-}$.  The principal regions and their nuclei are permuted amongst themeselves by $h$. For a complete treatment of principal regions and their nuclei, see~\cite[Subsection~6.5]{cc:hm}. The set $N=N_{1}\cup\cdots\cup N_{n}$ is called the \emph{nuclear invariant set}.

\begin{lemma}
  The invariant set $X_{h}$ is the union of the nuclear invariant set $N$ and the meager invariant set $X^{*}_{h}$.
\end{lemma}

We are going to define a ``tightening'' of $h|N$ (Definition~\ref{tight}) via the Nielsen-Thurston theory of compact surface automorphisms. Using this terminology, we can state the following basic result whose proof is given in Sections~\ref{nuclear} and~\ref{polyh}.

\begin{theorem}\label{finitegen}
If $h$ is a tight $\COO$ Handel-Miller monodromy, there is a minimal finite set of asymptotic cycles $\mu_{1},\mu_{r},\dots,\mu_{r}$ such that $\C'_{h}\subset H_{1}(\wh W)$ is the closed convex cone which is the topological closure in $H_{1}(\wh W)$ of the non-negative linerar combinations of $[\mu_{1}],[\mu_{2}],\dots,[\mu_{r}]$.
\end{theorem}

This has the following rather obvious corollary.

\begin{cor}\label{conepoly}
The cones $\C'_{h}$, $\C_{h}$ and $\C^{\kappa}_{h}$ are polyhedral.  The latter two are defined by the linear inequalities $[\mu_{i}]\ge0$ and have finitely many codimension~one faces defined by the linear equations $[\mu_{i}]=0$, $1\le i\le r$.
\end{cor}

Since $\XX_{h}$ is the union of $\XX_{N}$, the $\LL_{h}$-saturation of $N$, and $\XX^{*}_{h}$, the $\LL_{h}$-saturation of $X^{*}_{h}$, it will be enough to show that each of these sublaminations contributes finitely many generators. We begin with $\XX_{N}$.

\subsection{The nuclear invariant set ${N}$}\label{nuclear}
 
 There are cases.
\begin{enumerate}
\item A nucleus $N_{i}$ may be a disk.
\item  A nucleus $N_{i}$ may be an annulus or M\"obius strip.
\item A nucleus $N_{i}$ may have negative Euler characteristic.
\end{enumerate}

The $\LL_{h}$-saturation $\XX_{N_{i}}$ of a nucleus $N_{i}$ is the same   as the $\LL_{h}$-saturation  of the $h$-orbit of $N_{i}$.

\begin{lemma}
If $N_{i}$ is a disk, $\XX_{N_{i}}$ contributes at most one generator to the cone $\C'_{h}$. 
\end{lemma}

\begin{proof}
Indeed, this saturation of $N_{i}$ is a solid torus and has $1$-dimensional first homology.  
\end{proof}

\begin{rem}
We are not trying to produce a finite set of linearly independent generators.  In the case that $N_{i}$ is a disk, a spanning asymptotic class is represented by the closed, oriented $\LL_{h}$-orbit of a vertex of $\bd N_{i}$ and this orbit is also an asymptotic cycle for $\XX^{*}_{h}$.
\end{rem}

\begin{lemma}\label{AM}
If $N_{i}$ is an annulus or M\"obius strip, $\XX_{N_{i}}$ contributes at most two generators to the cone $\C'_{h}$.
\end{lemma}

\begin{proof}
Indeed, this saturation of $N_{i}$ has the homotopy type of the torus, hence has 2-dimensional homology. The corresponding cone of asymptotic classes is at most $2$-dimensional.
\end{proof}
 
We analyze Case~(3). Enumerate the nuclei with negative Euler characteristic as $N_{1},\dots,N_{m}$.  This set of nuclei is permuted by $h$.  We will isotope $h$ by an isotopy supported on the union  of these nuclei to ``Nielsen-Thurston'' form.

For each of the nuclei $N_{i}$, $1\le i\le m$,  let $p_{i}>0$ denote the $h$-period of $N_{i}$, the smallest positive integer such that $h^{p_{i}}(N_{i})=N_{i}$.   We want to apply the Nielsen-Thurston theory of automorphisms of surfaces~\cite{bca,FLP,HandT} to isotope $h$ by an isotopy supported in $N_{1} \cup N_{2}\cup\cdots  N_{m}$ to what  we will call Nielsen-Thurston form.  There is a small difficulty in that this theory was developed for smooth surfaces and $\bd N_{i}$ is only piecewise smooth.  We get that problem out of the way first.

 If $\sigma$ is a component of $\bd N_{i}$, there is a unique simple, closed, $2$-sided geodesic $\rho_{\sigma}\subset\intr N_{i}$ which cobounds an annulus with $\sigma$.  While $h(\rho_{\sigma})$ may not be $\rho_{h(\sigma)}$, they are freely homotopic to $h(\sigma)$, hence to each other, and a theorem of D.~B.~A.~Epstein~\cite[Theorem~2.1]{Epstein:isotopy} gives a smooth ambient isotopy $\phi$, compactly supported in $h(N_{i})$, such that $\phi(h(\rho_{\sigma}))=\rho_{h(\sigma)}$.  Thus, by replacing $h$ with the diffeomorphism $\phi\circ h$ isotopic to $h$ and agreeing with $h$ outside $N_{i}$, we can assume that  $h(\rho_{\sigma})=\rho_{h(\sigma)}$.  These isotopies can be chosen to have disjoint supports, hence we assume that $h$ permutes the set of loops $\rho_{\sigma}$ as $\sigma$ ranges over the boundary components of the nuclei in the $h$-orbit of $N_{i}$. 
Let $N_{i}'$ denote the connected subsurface of $N_{i}$ bounded by the circles $\rho_{\sigma}$, as $\sigma$ ranges over the components of $\bd N_{i}$. Thus, $N'_{i}$ has smooth boundary and is obtained from  $N_{i}$  by trimming off annular collars. To distinguish $N'_{i}$ from the nucleus $N_{i}$, we refer to it as the \emph{core} of the principal region and remark that $h^{p_{i}}  $ carries the core onto itself.

By the Nielsen-Thurston theory, there is an integer $k_{i}>0$ and an isotopy of $h^{p_{i}}|N'_{i}$ to a diffeomorphism $\theta_{i}$, such that there is a family of simple closed curves in $N'_{i}$, permuted among themselves by $\theta_{i}$,  splitting that surface into subsurfaces, each of which is invariant under $\theta_{i}^{k_{i}}$. On each of these subsurfaces, $\theta_{i}^{k_{i}}$ is  a pseudo-Anosov ``diffeomorphism'' or a periodic diffeomorphism.  In the pseudo-Anosov case, $\theta_{i}^{k_{i}}$ is a diffeomorphism outside of a finite set of multi-pronged singularities, at which it cannot even be $\CI$. Actually, the reducing circles need to be slightly thickened to a  family of $\theta_{i}^{k_{i}}$-invariant annuli in order that $\theta_{i}$ be smooth, or even continuous on $N'_{i}$.

\begin{rem}
While the pseudo-Anosov automorphism $\theta_{i}^{k_{i}}$ is  not quite a diffeomorphism, it is close enough for our purposes.  In fact, the $\LL_{h}$-saturation of the pseudo-Anosov subsurface  is of class $\COO$, the well understood structure of $\LL_{h}$ being smooth in the complement of the singular orbits and of class $\COO$ at each of these orbits~\cite[Appendix~B]{cc:LB}.  Following standard usage, we call it a pseudo-Anosov diffeomorphism.
\end{rem}

\begin{lemma}
There is an isotopy of $h$, supported in the nuclei, which produces the isotopies of $h^{p_{i}}|N'_{i}$ in the Nielsen-Thurston theory.
\end{lemma}

\begin{proof}
First of all, we must deal with the fact that $\theta_{i}$ and $h^{p_{i}}$ may not agree on $\bd N'_{i}$.
Fix a choice of $N_{i}$ and, for notational simplicity, set $p_{i}=p$,  $\theta_{i}=\theta$ and re-index the $h$-orbit of $N_{i}$ as $N_{0},N_{1},\dots,N_{p-1}$, where the index is taken mod $p$. Then $h^{p}$ and $\theta$ are  isotopic on each boundary circle $\rho_{\sigma}$, hence $h^{-p}\circ\theta$ is isotopic to the identity on $\rho_{\sigma}$, as this curve varies over the components of $\bd N_{j}$, $0\le j<p$.  Using this isotopy, extend $h^{-p}\circ\theta$ over all of the nucleus $N_{j}$ to a diffeomorphism $\phi$, isotopic to the identity and equal to  the identity in a neighborhood of $\bd N_{j}$.  Then $h^{p}\circ\phi$  restricts to $\theta$ on the core $N'_{j}$ and agrees with $h^{p}$ near the boundary of the nucleus $N_{j}$.  Rename $h^{p}\circ\phi$, denoting it by $\theta$.

Let $h_{j}=h|N_{j}:N_{j}\to N_{j+1}$, $0\le j\le p$. Set
$$
h'_{0}=h_{1}^{-1}\circ h_{2}^{-1}\circ\cdots\circ h_{p-1}^{-1}\circ\theta,
$$
a diffeomorphism isotopic to $h_{0}$.  Replacing $h_{0}$ with $h'_{0}$ changes $h$ by an isotopy supported in the given nucleus.  We do not modify $h_{j}$, $1\le j\le p-1$.  One easily checks that, with this new choice of $h$, $h^{p}|N_{j}$ is in Nielsen-Thurston form, $0\le j\le p-1$.
\end{proof}

\begin{defn}\label{tight}
The Handel-Miller monodromy, modified on $N$ as above, is said to be \emph{tight}.
\end{defn}

From now on, we assume that $h$ is tight. There is a finite, $h$-invariant family of annuli in the Nielsen-Thurston decomposition, those separating the pseudo-Anosov and/or periodic pieces and the collars that were trimmed off of the $N_{i}$'s to produce the smooth cores $N'_{i}$. Exactly as in the proof of Lemma~\ref{AM}, we obtain the following.

\begin{lemma}
The $\LL_{h}$-saturation of the finite, $h$-invariant family of annuli contributes at most finitely many generators to the cone $\C'_{h}$.
\end{lemma}

\begin{lemma}
Let $S\subset N'_{i}$ be such that $h^{k_{i}p_{i}}|S$ is periodic.  Then the $\LL_{h}$-saturation of $S$ contributes only one generator to the cone $\C'_{h}$.
\end{lemma}

\begin{proof}
A  theorem of Epstein, concerning foliations of compact $3$-manifolds by circles,  asserts that the foliation is a Seifert fibration~\cite{Ep:3}. The saturation of $S$ is a compact 3-manifold with boundary that is foliated by circles.    It follows that every fiber is homologically a rational multiple of every other fiber, completing the proof. 
\end{proof}

Thus, to complete the proof of Theorem~\ref{finitegen}, we must show that the saturation of the pseudo-Anosov components and of the meager invariant set contribute finitely many generators.  In both cases, the proof is the same, making use of  Markov dynamics.  This will be carried out in the next section.

\subsection{Markov dynamics and generators of $\C'_{h}$}\label{polyh}
We fix a tight Handel-Miller monodromy diffeomorphism $h:L\to L$.

 If  for some positive integer $q$, $h^{q}:S\to S$ is pseudo-Anosov, where $S$ is a subsurface of a core, it is well known that there is a Markov partition for the dynamical system generated by $h^{q}$~\cite[Section~6]{bca} and~\cite[Expos\'e~9]{FLP}.  This is a family of ``rectangles'' $R_{1},R_{2},\dots,R_{r}$ with disjoint interiors, each with a pair of opposite sides called ``top'' and ``bottom'' and a pair of opposite sides called ``left'' and ``right''.  These have the property that, for $1\le i,j\le r$, $h^{q}(R_{i})$ crosses $R_{j}$, if at all, just once and does so through the top and bottom.  Similarly, $h^{-q}(R_{i})$ crosses $R_{j}$ at most once and does so through the left and right sides.  Similarly, the dynamical system $h:X^{*}_{h}\to X^{*}_{h}$ admits a Markov partition~\cite[Section~9]{cc:hm}. In this latter case, the top and bottom sides of the rectangles are arcs in leaves of $\Lambda_{-}$ and the left and right sides are arcs in $\Lambda_{+}$.  We allow rectangles that degenerate to a single arc or even a point.  In the pseudo-Anosov case, the Markov rectangles are non-degenerate.  In the following discussion $g:X\to X$ represents either $h^{q}:S\to S$ or $h:X^{*}_{h}\to X^{*}_{h}$.  We let $\XX$ denote the $\LL_{h}$-saturation of $X$, a compact lamination.  Our goal in this section is to prove the following.

\begin{theorem}\label{hull}
The asymptotic cycles for $\XX$ contribute finitely many generators to the cone $\C_{h}$.
\end{theorem}

Together with the results of the previous section, this will complete the proof of Theorem~\ref{finitegen}.

There is an $r\x r$  matrix $A=[a_{ij}]$ of 0's and 1's encoding which letter can follow which. That is, $j$ can follow $i$ if and only if $a_{ij}=1$. The  set of all allowable bi-infinite sequences is denoted by $\mathcal{S}_{A}$ and the so-called subshift of finite type $ \sigma_{A}:\mathcal{S}_{A}\to\mathcal{S}_{A}$ shifts each sequence one step to the right.  There is a compact, totally disconnected (usually Cantor) topology on $\mathcal{S}_{A}$ and $g:X\to X$ is semi-conjugate to $\sigma_{A}$ as we now describe. 
An element $\iota=(i_{k})_{k\in\Z}\in\mathcal{S}_{A}$ represents a unique point $x_{ \iota}\in R_{i_{0}}\cap X$ such that $g^{k}(x_{ \iota})\in R_{i_{k}}$, $\forall k\in\Z$.  In terms of the lamination $\XX$, this means that the leaf issuing from $x_{ \iota}$ meets $L$ successively in $R_{i_{0}},R_{i_{1}},\dots,R_{i_{k}},\dots$ in forward time, with a corresponding statement for backward time.  While each $\iota\in\mathcal{S}_{A}$ encodes a unique point of  $X$, some points of $X$  may have finitely many such representatives.  The problem is that distinct Markov rectangles  may  meet along parts of their boundaries.  Thus the map $\iota\mapsto x_{ \iota}$ is finite to one, defining a semi-conjugacy of $ \sigma_{A}$ to $g:X\to X$.  We omit details since the theory of Markov partitions is a well understood part of dynamical system theory.

\begin{rem}
In the case of $h:X^{*}_{h}\to X^{*}_{h}$, the rectangles can be chosen to be disjoint and the semi-conjugacy is an honest conjugacy.
\end{rem}

The periodic elements of $\mathcal{S}_{A}$ are those carried to themselves by some power $ \sigma_{A}^{p}$, $p\ge1$. These correspond to closed leaves in $\XX$. The substring $( i_{0}, i_{1},\dots, i_{p-1})$ of a periodic sequence $ \iota$, $ \sigma_{A}^{p}( \iota)={ \iota}$, where $p\ge1$ is minimal, will be called the period of $ \iota$.  The substring $(i_{0},i_{1},\dots,i_{p-1},i_{0})$ will be called a periodic string.  If no proper substring of a period is a periodic string, we say that the period is \emph{minimal}.
Since there are only finitely many distinct entries occurring    in the sequences $\iota\in\mathcal{S}_{A}$, it is evident that there are
only finitely many minimal periods.  Those closed leaves $\gamma $ of
$\XX$ that correspond to minimal periods in the symbolic system will
be called minimal loops in $\XX$ and denoted by  $ \gamma_{1}, \gamma_{2},\dots, \gamma_{n}$.

Let $\Phi_{t}$ denote the flow on $\wh W$ that stabilizes $\tb\wh W$ pointwise, has flow lines in $W$ coinciding with the leaves of $\LL_{h}$ and is parametrized so as to preserve $\FF|W$ and so that $\Phi_{1}|L=h$.

Let $\iota =(i_{k})_{k=-\infty}^{\infty}\in \mathcal{S} _{A}$ and suppose
that $i_{q}=i_{0}$ for some $q>0$. Let $x\in R_{\iota }=\bigcap_{j=-\infty}^{\infty}R_{i_{j}}$.  Then there
is a corresponding singular cycle $\Gamma _{q}$ formed from the
orbit segment $\gamma _{q}=\{\Varphi _{t}(x)\}_{0\le t\le q}$ and an
arc $\tau \subset R_{i_{0}}$ from $\Varphi _{q}(x)=h^{q}(x) \text{ to }
x$.  Also, since $i_{q}=i_{0}$, there is a periodic element $\iota
'\in
\Sigma _{A} $ with period $i_{0},\dots,i_{q-1}$ and a corresponding
closed leaf $\Gamma _{\iota '}=\Gamma '$ of $\XX$.

\begin{lemma}\label{periodic} 
The singular cycle $\Gamma _{q}$ and closed leaf $\,\Gamma '$, obtained as
above, are homologous in $\wh W$. In particular, the homology class of $\Gamma _{q}$
depends only on the periodic element $\iota '$.
\end{lemma}

\begin{proof}
The loop $\Gamma '$ is the orbit segment $\{\Varphi _{t}(y)\}_{0\le
t\le q}$, for a periodic point $$y\in R_{i_{0}}\cap
h^{-1}(R_{i_{1}})\cap\dots\cap h^{-q}(R_{i_{q}})=R'.$$ Remark that
$x\in R'$ also.  Let $\tau '$ be an arc in the rectangle $R'$ from $x$ to $y$ and set $\tau ''=h^{q}(\tau ')$, an arc in
$h^{q}(R')$ from $h^{q}(x)$ to $y$.  Since $i_{q}=i_{0}$,
$h^{q}(R')\subset R_{i_{0}}$ and the cycle $\tau +\tau '-\tau ''$ in the
rectangle $R_{i_{0}}$ is homologous to 0.  That is, we can replace the
cycle $\Gamma _{q}=
\gamma _{q}+ \tau $ by the homologous cycle $\gamma _{q}-\tau '+\tau
''$.  Finally, a homology between this cycle and $\Gamma '$ is
given by the map $$H:[0,1]\times [0,q]\ra \wh W,$$  defined by
parametrizing $\tau '$ on $[0,1]$ and setting
$$H(s,t)=\Varphi _{t}(\tau '(s)).$$
\end{proof}

\begin{cor}\label{minloops}
Every closed leaf $\Gamma $ of $\XX$ is homologous in $\wh W$ to a linear
combination of the minimal loops in $\XX$ with non--negative integer
coefficients.
\end{cor}

\begin{proof}
The closed leaf $\Gamma$ corresponds to a period $(i_{0},\dots,i_{q-1})$. If this period is minimal, we are done.  Otherwise, after a cyclic permutation, we can assume that the period is of the form  $(i_{0},i_{1},\dots,i_{p}=i_{0},i_{p+1},\dots,i_{q-1})$.  We then see that $\Gamma$ is homologous to the sum of two loops, one being the arc $\gamma$ of $\Gamma$ corresponding to the periodic string $(i_{0},\dots,i_{p}=i_{0})$ followed by an arc $\tau$ from the endpoint of $\gamma$ to its initial point, and one being $-\tau+\gamma'$, where $\gamma'$ is the subarc of $\Gamma$ corresponding to the periodic string $(i_{0},i_{p+1},\dots,i_{q-1},i_{q}=i_{0})$.  By Lemma~\ref{periodic}, both $\gamma+\tau$ and $-\tau+\gamma'$ are homologous to closed orbits corresponding to periods strictly shorter than $(i_{0},\dots,i_{q-1})$.  Thus, finite iteration of this procedure proves the corollary.
\end{proof}

Let $\iota=(i_{k})_{k=-\infty}^{\infty},\iota'=(i'_{k})_{k=-\infty}^{\infty}\in\mathcal{S}_{A}$ and suppose that 
$$
(i_{0},i_{1},\dots,i_{q})=(i'_{0},i'_{1},\dots,i'_{q}),
$$  not necessarily a period.  Let $x=x_{\iota}$ and $x'=x_{\iota'}$.  Both of these points are in 
$$
R' = R_{i_{0}}\cap
h^{-1}(R_{i_{1}})\cap\dots\cap h^{-q}(R_{i_{q}}).
$$  Choose a path $\tau$ in $R_{i_{0}}$ from $x'$ to $x$ and a path  $\tau'\subset R_{i_{q}}$ from $h^{q}(x')$ to $h^{q}(x)$.  Consider the orbit segments $\Gamma=\{\Phi_{t}(x)\}_{t=0}^{q}$ and $\Gamma'=\{\Phi_{t}(x')\}_{t=0}^{q}$. 
Let $K$ be an upper bound of the diameters of $R_{i}$, $1\le i\le n$.  Then the paths $\tau$ and $\tau'$  can always be chosen to have length less than $K$.

 The following is proven analogously to Lemma~\ref{periodic}.

\begin{lemma}\label{homologouschains}
The singular chains $\Gamma'$ and $\tau+\Gamma-\tau'$ are homologous.  In particular, for each closed $1$-form $\eta$ on $\wh W$, 
$$
\int_{\Gamma'}\eta=\int_{\tau+\Gamma-\tau'}\eta.
$$  Here, the paths $\tau$ and $\tau'$ have length less than $K$.
\end{lemma}

\begin{prop}\label{combsofminloops}
Every homology direction can be arbitrarily well approximated by nonnegative linear combinations of the minimal loops.  
\end{prop}

\begin{proof}
 Let $\Gamma=\{\Phi(t)(x)\}_{t=-\infty}^{\infty}$ be an orbit and suppose that $x$ corresponds to the symbol $\iota=\{i_{r}\}_{r=-\infty}^{\infty}$.  By a suitable shift, we can assume that $i_{0}$ occurs infinitely often in forward time in this symbol.  Consequently, for each index $i$ in $\iota$, there is a positive integer $k_{i}$ such the the $(i,i_{0})$-entry in $A^{k_{i}}$ is strictly positive.  Let $k$ be the largest of the $k_{i}$.  Thus, given a substring $(i_{0},i_{1},\dots,i_{q})$ of $\iota$, there is a periodic element $\iota'\in\Sigma_{A}$ with period $(i_{0},i_{1},\dots,i_{q},i_{q+1},\dots,i_{q+s})$, where $s\le k$. Let $\Gamma'_{q}$ denote the corresponding periodic orbit.  If we parametrize the flow $\Phi_{t}$ by the invariant measure for $\FF$ of period 1,  then the length of the segment $\Gamma_{q}$ of $\Gamma$ corresponding to the string $(i_{0},i_{1},\dots,i_{q})$ is $q$.  Choosing a suitable sequence $q\uparrow\infty$, we obtain the general homology direction $$\mu=\lim_{q\to\infty}\frac{1}{q}\int_{\Gamma_{q}}.$$  Passing to a subsequence, we also obtain a cycle $$
\mu'=\lim_{q\to\infty}\frac{1}{q}\int_{\Gamma'_{q}}.
$$  Since $s$ is bounded independently of $q$, Lemma~\ref{homologouschains} implies that $\mu$ and $\mu'$ agree on all closed 1-forms, and so $\mu'$ is a cycle homologous to $\mu$ (both in $(\DD'_{*},\bd)$ and $(\EE'_{*},\bd)$).  Corollary~\ref{minloops} then implies the assertion.
\end{proof}

By Lemma~\ref{span}, Theorem~\ref{hull} follows, hence the proof of Theorem~\ref{finitegen} is complete.

\subsection{Uniqueness of the Handel-Miller  cone}

Each periodic piece in the Nielsen-Thurston decomposition of the nuclei contributes just a ray of homology classes  to $\C'_{h}$.  Each pseudo-Anosov piece contributes a closed, convex subcone.  Although the choice of the restriction of $h$ to a pseudo-Anosov pieces  is not unique, any two choices $h$ and $h'$ are related by $(h')^{p}=\phi\circ h^{p}\circ\phi^{-1}$, where $p$ is the power of $h$ leaving the pseudo-Anosov piece invariant and $\phi:L\to L$ is a homeomorphism isotopic to the identity and supported on the pseudo-Anosov piece (cf.~\cite{th:surfaces}). This also holds, by the same reference, for the periodic pieces. By an argument entirely analogous to the proof of Corollary~\ref{conj}, we obtain the following.
  
  \begin{prop}\label{Nindep}
  The Handel-Miller  cone $\C'_{h}=\C'_{h^{p}}$ is independent of the allowable choices of $h$ on the pseudo-Anosov and periodic pieces in the nuclei.
  \end{prop}
  
  \begin{rem}
   In particular, not only are the cones independent of the allowable choices of $h$ but also the homology directions associated to the map $h^{p}$ span the same cone as those of $h$. 
  \end{rem}

\begin{theorem}\label{hmindep}
The Handel-Miller  cone $\C'_{h}\subset H_{1}(\wh{W})$ is independent of the choice of the tight Handel-Miller representative of the isotopy class of the endperiodic monodromy of $L$.
\end{theorem}

\begin{proof}
Let $h$ and $g$ be two such choices.  
By~\cite[Theorem~10.10]{cc:hm}, the laminations $\Lambda_{\pm}$ associated to $h$ and the laminations $\Pi_{\pm}$ associated to $g$ are simultaneously ambiently isotopic.  That is $\Pi_{\pm}=\phi(\Lambda_{\pm})$ where $\phi:L\to L$ is a homeomorphism isotopic to the identity.  Let $X^{*}_{h}$ and $X^{*}_{g}$ denote the meager invariant sets for the respective laminations.

While the choices of $h$ and $g$ associated to these laminations is not unique, the restrictions $h|X^{*}_{h}$ and $g|X^{*}_{g}$ are unique.  Combined with Proposition~\ref{Nindep}, this allows us to assume that 
$\phi^{-1}\circ g\circ\phi=h$ on the full invariant set.  By Corollary~\ref{conj},  $\C'_{h}=\C'_{g}$.
\end{proof}

\begin{rem}
Because of this theorem, we will  denote the cone $\C'_{h}$ by $\C'_{\FF}$ and the dual $\hk(\wh{W})$-cohomology cone by  $\Ck_{\FF}$ and  the dual $\hI(\wh{W})$-cohomology cone by  $\C_{\FF}$. Note that $h$ is tight (and smooth) Handel-Miller monodromy.
\end{rem}

\begin{defn}

The cohomology cones   $\C_{\FF}$ and $\Ck_{\FF}$ are called \emph{Handel-Miller foliation cones}.

\end{defn}

\begin{rem}

The cohomology cones   $\C_{\FF}$ and $\Ck_{\FF}$ are the cones in Theroems~\ref{cone} and~\ref{conesmooth}.

\end{rem}

\subsection{The proof of Theroem~\ref{conesmooth}}

We showed in Theorem~\ref{hmindep} that the Handel-Miller  cones $\C'_{\FF}\subset H_{1}(\wh{W})$ are well defined. It follows that the Handel-Miller foliation cones $\Ck_{\FF}\subset\hk(\wh{W})$ are well defined. By Corollary~\ref{conepoly}, the Handel-Miller foliation cones are closed, convex, and polyhedral with  finitely many faces of codimension one.

Suppose $\wh\HH$ is an admissible foliation of $\wh W$ and $\wh\LL$ is a stongly transverse one-dimensional foliation. Then by Theorem~\ref{folforms} the open cone $\intr\C_{\XX}\subset\hI(\wh W)$ consists of foliated forms transverse to $\LL$.  Thus, there exists a foliation of relative depth one in $\C_{\XX}$ with monodromy $f$, $\LL = \LL_{f}$ and $\Ck_{\XX} = \Ck_{f}$ which, by Theorem~\ref{max}, is contained in some Handel-Miller foliation cone $\Ck_{\FF}$. Thus, $\wh\HH\in\Ck_{\FF}$. Thus, the $\CO$ isotopy classes of admissible foliations $\wh\HH$ of $\wh W$ that are  trivial in the arms are in natural  correspondence with the rays out of the origin in the interiors of the Handel-Miller foliation cones. We will prove that this correspondance is one-to-one in~\cite{cc:isorel}.

To prove Theroem~\ref{conesmooth}, it remains to show that the foliation cones have disjoint interiors, which we do in Theorem~\ref{max}, and that the Handel-Miller foliation cones are finite in number, which we do in Section~\ref{conefinite}.

\subsubsection{Maximality of the Handel-Miller  foliation cones}   Our next goal is to show that, if $g$ is an endperiodic map  in the  isotopy class of $h$, then $\C_{g}^{\kappa}\subseteq\Ck_{\FF}$. In fact, we will show that if $g$ is a monodromy map for \emph{any} smooth fibration $\GG$ of $W$, then either $(\intr\Ck_{g})\cap\Ck_{\FF}=\0$, or $\Ck_{g}\subseteq\Ck_{\FF}$.

We begin with an analysis of the group $G=H^{1}_{\kappa}(\wh{W})\cap H^{1}(\wh{W};\Z)$.  This will play the role of the integer lattice in $\hk(\wh{W})$.  The ``rational rays'' in this vector space will be those rays issuing from the origin that meet $G$ in nonzero points.  It will be crucial that the union of the rational rays be dense in $\hk(\wh{W})$.

Recall from Subsection~\ref{tops} that the topology on $\hk(\wh{W})=\ura{\lim}\hk(K_{i}^{\circ})$  can be taken to be the weak topology.  

By $\bd K_{i}$ we will mean the relative boundary of $K_{i}$ in $\wh W$. It consists of finitely many disjoint rectangles and/or annuli. Each annular component interfaces with an arm of the octopus decomposition and is $\LL_{h}$-saturated. Each remaining arm  interfaces with $K_{i}$ along a union of some of the rectangular components. No component can be a M\"obius strip since it will be fibered by \emph{oriented} intervals that are arcs of $\LL_{h}$.  Recall that $K_{i}^{\circ}=K_{i}\sm\bd K_{i}$.

\begin{lemma}\label{vanonbd'}
The subspace $\hk(K_{i}^{\circ})\subset\hI(K_{i})$ consists of those classes which restrict to $0$ in $\hI(\bd K_{i})$.
\end{lemma}

\begin{proof}
If $[\omega]\in\hI(K_{i})$ restricts to $0$ on $\bd K_{i}$, then $\omega|\bd K_{i}$ is exact, hence is also exact in a normal neighborhood $N$ of $\bd K_{i}$. That is, $\omega|N=d\lambda$ for a smooth function $\lambda$ which can be damped off to $0$ and extended by $0$ to a smooth function $\lambda$ on all of $K_{i}$.  Then $\omega-d\lambda$ has compact support in $K_{i}^{\circ}$.  The converse is immediate.
\end{proof}

Set $G_{i}=\hk(K_{i}^{\circ})\cap\hI(K_{i};\Z)$, a finitely generated, free abelian group. 

\begin{lemma}\label{fullattice}
The subgroup $G_{i}\subset\hk(K_{i}^{\circ})$ is a full lattice subgroup of \upn{(}i.e., spans\upn{)}  this finite dimensional vector space.
\end{lemma}

\begin{proof}
The core loops of the annular components of $\bd K_{i}$  represent elements (not necessarily linearly independent) of the finitely generated, free abelian group $\Gamma_{i}=H_{1}(K_{i};\Z)/\tors$. We call them ``peripheral'' elements. Let $\Gamma'_{i}\subseteq\Gamma_{i}$ be the smallest direct summand in this lattice that contains the peripheral elements.  In particular, by Lemma~\ref{vanonbd'}, the classes in $\hI(K_{i})$ which vanish on $\Gamma'_{i}$ are exactly the elements of $\hk(K_{i}^{\circ})$.   If $[\sigma_{1}],[\sigma_{2}],\dots,[\sigma_{r}]$ are a basis of $\Gamma'_{i}$, then extend to an ordered basis $([\sigma_{1}],\dots,[\sigma_{r}],[\sigma_{r+1}],\dots,[\sigma_{q}])$  of $\Gamma_{i}$, hence of $H_{1}(K_{i})$, Let $([\omega_{1}],\dots,[\omega_{r})],[\omega_{r+1}],\dots,[\omega_{q}])$ be the dual basis of $\hI(K_{i})$. Then, by Lemma~\ref{vanonbd'}, the classes $[\omega_{j}]$ lie in $G_{i}$, $r+1\le j\le q$, and form a basis of $\hk(K_{i}^{\circ})$.
\end{proof}

It is standard that the rays in the finite dimensional vector space $\hk(K_{i}^{\circ})$ which meet the full lattice $G_{i}$ at nonzero points  (the rational rays) unite to form a dense subset of that vector space.  Furthermore, note that $$G_{0}\subseteq G_{1}\subseteq\cdots\subseteq G_{i}\subseteq\cdots$$ has increasing union $G\subset\hk(\wh{W})$.  An appeal to the definition of the weak topology makes the following clear.

\begin{cor}\label{ratdense}
The union of the rational rays is dense in $\hk(\wh{W})$.
\end{cor}

 \begin{prop}\label{samecone}
The rational foliated ray $\left<\GG\right>$ lies in $\intr\Ck_{\FF}$  if and only if $\Ck_{\GG}=\Ck_{\FF}$.
\end{prop}

\begin{proof}
Suppose that $\<\GG\>\subset\intr\Ck_{\FF}$.
By Theorem~\ref{folforms}, we take the foliation $\GG$  to be transverse to $\LL_{h}$.  By the ``transfer theorem''~\cite[Theorem~12.5]{cc:hm}, $\LL_{h}$ induces Handel-Miller monodromy $g$ on each leaf of $\GG|W$.  A comment is needed since, in~\cite{cc:hm}, it was not required that Handel-Miller monodromy be exactly Nielsen-Thurston on $N'$. But, as the transfer theorem is obvious on the Nielsen-Thurston periodic pieces and Fried~\cite{fried} proves the transfer theorem for the Nielsen-Thurston pseudo-Anosov pieces, the discussion in Section~12 of~\cite{cc:hm}  is easily augmented to accomodate that requirement.  Thus, the cones $\Ck_{\GG}$ and $\Ck_{\FF}$ are determined by the same core lamination $\XX_{g}=\XX_{h}$ and so are identical.  For the converse, $\Ck_{\GG}=\Ck_{\FF}$, for rational foliated rays $\<\GG\>$ and $\<\FF\>$, clearly implies that $\<\GG\>\subset\intr\Ck_{\FF}$.
\end{proof}

\begin{cor}\label{notinbd}
No rational foliated ray is contained in $\bd\Ck_{\FF}$.
\end{cor}

\begin{proof}
If there is a rational foliated ray $\<\GG\>\subset\bd\Ck_{\FF}$,  then $\intr\Ck_{\GG}\cap\intr\Ck_{\FF}\ne\0$.  By Corollary~\ref{ratdense}, there is a rational foliated ray $\<\HH\>\subset\intr\Ck_{\GG}\cap\intr\Ck_{\FF}$.  By Proposition~\ref{samecone}, we see that $\Ck_{\GG}=\Ck_{\HH}=\Ck_{\FF}$.  That is, $\<\GG\>\subset\intr\Ck_{\FF}$, contrary to our hypothesis.
\end{proof}

The boundary $\bd\Ck_{\FF}$ is made up of  $r$ top faces $F_{1},\dots,F_{r}$, where $F_{i}$ is a convex, polyhedral cone with nonempty (relative) interior in the hyperplane $[ \mu_{i}]=0$.  

\begin{lemma}\label{denseinFi}
Each $F_{i}$ contains a dense family of rays that meet nontrivial points of the integer lattice $\hk(\wh W;\Z)$.  
\end{lemma}

\begin{proof}
Since $[ \mu_{i}]$ is an integral homology class, it takes integer values on the integer lattice.  The corresponding claim is standard in the finite dimensional spaces $\hk(K_{j}^{\circ})$ and remains true in the direct limit by the definition of the weak topology.
\end{proof}

\begin{theorem}\label{max}
If $g$ is a monodromy map \upn{(}endperiodic\upn{)} for a fibration $\GG$ of $W$, then either $(\intr\Ck_{g})\cap\Ck_{\FF}=\0$, or $\Ck_{g}\subseteq\Ck_{\FF}$.  In particular, $\Ck_{\FF}=\Ck_{h}$ is the maximal foliation cone for monodromies in the isotopy class of $h$  and two maximal foliation cones either coincide or have disjoint interiors.
\end{theorem}

\begin{proof}
If $(\intr\Ck_{g})\cap\Ck_{\FF}\ne\0$ and $\Ck_{g}\not\subseteq\Ck_{\FF}$, then Lemma~\ref{denseinFi} implies that there is a rational foliated ray in $\bd\Ck_{\FF}$, contradicting Corollary~\ref{notinbd}.
\end{proof}

\begin{rem}
Correspondingly, the dual homology cone $\C'_{\FF}=\C'_{h}$ is the minimal $\C'_{g}$ for all monodromies $g$ isotopic to $h$.  In this sense, we can say that the tight Handel-Miller monodromy has the ``tightest'' dynamics in its isotopy class.
\end{rem}

\subsubsection{Finiteness of the set of Handel-Miller  foliation cones}\label{conefinite} 

The nucleus $K_{0}$ of the octopus decomposition of $\wh W$ cuts off finitely many arms which are of the form $B\x I$, where $B\subset \tb \wh W$ is connected and noncompact.  All of our foliations $\GG$ that appropriately extend $\FF|(M\sm W)$, when restricted to an arm, are transverse to the $I$ fibers.  Thus,  for all the appropriate foliations $\GG$ with Handel-Miller monodromy $g$, the core lamination $\XX_{g}$ lies in $K_{0}$.  The cone $\C_{\GG}^{\kappa}$ is defined by  inequalities $[\gamma_{i}]\ge0$, $1\le i\le r$, where each $\gamma_{i}$ is  a periodic orbit in $\XX_{g}$.  These same inequalities define the cone $\C_{\GG|K_{0}}\subset H^{1}(K_{0})$ determined by the depth one foliation $\GG|K_{0}$ of the sutured manifold $K_{0}$.  This sets up a one-to-one correspondence between the set $\KK_{W}$ of Handel-Miller foliation cones in $\hk(\wh{W})$ and a subset of the set of Handel-Miller foliation cones in $H^{1}(K_{0})$. By~\cite[Theorem~6.4]{cc:cone}, it follows that there are only finitely many of these cones.

The proof of Theorem~\ref{conesmooth}  is complete.

\section{Foliations of class $\COO$}\label{HMcones1}

In this section, we prove Theorem~\ref{cone}. We relax the smoothness condition on the foliation $\FF$ of $M$ to $\COO$ and analyze the cones $\C_{i}\subset\hi(\wh W)$.

The isomorphism $\hI(\wh W)=\ula{\lim}\,\hI(K_{i})$ is induced by the natural homomorphisms $\psi_{i}:\hI(\wh{W})\to\hI(K_{i})$. Recall from Subsection~\ref{tops} that the topology on $\hI(\wh{W})$ is the standard inverse limit topology, relativized from the Tychonov topology via the inclusion $$\ula{\lim}\hI(K_{i})\subset\hI(K_{0})\x\hI(K_{1})\x\cdots\x\hI(K_{i})\x\cdots.$$  Since $\psi_{i}$ is induced by projection of the product onto its $i$th factor, it is continuous.

Let $V_{i}\subset\hI(K_{i})$ denote the image of $\psi_{i}$ and remark that $\hI(\wh{W})=\ula{\lim}V_{i}$. Also note that, since $\hI(\wh{W})=\hI(\wh{W};\Z)\otimes\R$, the integer lattice $\hI(\wh{W};\Z)$ is carried onto a full lattice subgroup of $V_{i}$, $i\ge0$. Indeed, it is carried onto $V_{i}\cap\hI(K_{i};\Z)$ and this must be a full lattice subgroup since $\psi_{i}$ surjects onto $V_{i}$.  Recall that the rays issuing from the origin in $\hI(\wh{W})$ that meet nonzero integer lattice points are called  rational rays.

\begin{lemma}\label{uniondense}
The union of rational rays in $\hI(\wh{W})$ is everywhere dense.
\end{lemma}

\begin{proof}
The open sets in $\hI(\wh{W})=\ula{\lim}V_{i}$ are unions of sets of the form $$U=\ula{\lim}V_{i}\cap (U_{0}\x U_{1}\x\cdots U_{n}\x V_{n+1}\x V_{n+2}\x\cdots),$$  where $U_{i}\subseteq V_{i}$ is open, $0\le i\le n$.  If this set is nonempty, we must prove that it meets a rational ray.  Let $\theta_{i}:V_{n}\to V_{i}$ be the natural surjection, $0\le i<n$.  Then the set
$$
Y=U_{n}\cap\theta_{n-1}^{-1}(U_{n-1})\cap\cdots\cap\theta_{0}^{-1}(U_{0})
$$ is open and the assumption that $U\ne\0$ implies that $Y\ne\0$.  Since $V_{n}$ is finite dimensional, the rational rays do have dense union there, hence we select such a ray $\rho$ that meets $Y$.  There is a rational ray $\rho'$ in $\hI(\wh{W})$ which is mapped onto $\rho$ by $\psi_{n}$.  Viewing $\rho'$ in $\ula{\lim}V_{i}$, one sees that it meets $U$. 
\end{proof}

Let $ \zeta\in H_{1}(\wh{W})$.  This class is represented by a singular cycle which necessarily lives in $K_{k}$, for large enough $k\ge0$. Then $\zeta:\hI(K_{k})\to\R$ is a continuous linear functional. We can define $\zeta:\hI(\wh{W})\to\R$ by the composition
$$
\hI(\wh{W})\xra{\psi_{k}}\hI(K_{k})\xra{\zeta}\R,
$$
remarking that this is independent of the choice of large enough $k$ and defines a continuous linear functional.

We assume $\FF$ is only of class $\COO$.    
The reader will note that the lack of smoothness is only for the foliations in $M$, the foliation in each $\wh{W}$ being itself of class $\Ci$. 
We consider $\FF|\wh W$.  While this foliation may never trivialize outside any $K_{i}$, $\FF|K_{i}$ is a depth one foliation and our previous discussion can be applied to it. Note that, since we are working in the compact manifold $K_{i}$ rather  than $K_{i}^{\circ}$, all differential forms are compactly supported. We can assume that the arms are the connected components of $\wh{W}\sm K_{0}^{\circ}$, hence that $\FF$ is transverse to the interval fibers in these arms.  

The above discussion works in dimensions $\ge3$ without restrictions on the topology of leaves.  We now require $\dim M=3$ and that semiproper leaves satisfy the assumption in Subsection~\ref{restrict}. Handel-Miller theory~\cite{cc:hm} does not apply directly to the foliations of $\wh W$ that are fibrations extending the foliation $\FF|(M\sm W)$.  Indeed, such a foliation may be nontrivial outside every $K_{i}$, hence  the monodromy will not be endperiodic. It will be necessary to work in the progressively expanding sutured manifolds $K_{i}$, in each of which the endperiodic theory works fine, in order to produce the Handel-Miller foliation cones in $\hI(\wh W)=\ula{\lim}\,\hI(K_{i})$. We only need note that, by the assumption that no semiproper leaf is contractible or homotopic to the circle, we can guarantee that each component of $\tb K_{i}$ has negative Euler characteristic.

We can choose the monodromy $h$ so that, in the arms,  it is given by the flow along these interval fibers and is tight (and smooth) Handel-Miller monodromy for $\FF|K_{0}$, hence also for $\FF|K_{i}$, $i\ge0$.  (One should note that a leaf of $\FF$ may intersect $K_{i}$ in a finite family of leaves of $\FF|K_{i}$, but this really doesn't matter.)  Thus, the core lamination $\XX_{h}$ of $\LL_{h}|K_{i}$ lives in $K_{0}$ and is independent of $i$.  Recall from Theorem~\ref{finitegen} and Corollary~\ref{conepoly} that this lamination gives a set $[ \mu_{1}],[ \mu_{2}],\dots,[ \mu_{r}]\in H_{1}(\wh{W})$, hence the linear inequalities $[ \mu_{i}]\ge0$ define a polyhedral cone in $\hI(\wh{W})$, the  cone of foliations we have designated by $\C_{h}$ and now designate by $\C_{\FF}$.

\subsection{The proof of Theroem~\ref{cone}}

We showed in Theorem~\ref{hmindep} that the Handel-Miller  cones $\C'_{\FF}\subset H_{1}(\wh{W})$ are well defined. It follows that the Handel-Miller foliation cones $\C_{\FF}\subset\hI(\wh{W})$ are well defined.   By Corollary~\ref{conepoly}, the Handel-Miller foliation cones are closed, convex, and polyhedral with  finitely many faces of codimension one.

Suppose $\wh\HH$ is an admissible foliation of $\wh W$ and $\wh\LL$ is a stongly transverse one-dimensional foliation. Then, by Theorem~\ref{folforms}, the open cone $\intr\C_{\XX}\subset\hI(\wh W)$ consists of foliated forms. Thus, there exists a foliation of relative depth one in $\C_{\XX}$ with monodromy $f$. Therefore, $\LL = \LL_{f}$ and $\C_{\XX} = \C_{f}$ which, by Theorem~\ref{max0} is contained in some Handel-Miller foliation cone $\C_{\FF}$. Thus, $\wh\HH\in\C_{\FF}$. Thus, the $\CO$ isotopy classes of admissible foliations $\wh\HH$ of $\wh W$ that are  trivial in the arms are in natural  correspondence with the rays out of the origin in the interiors of the Handel-Miller foliation cones. We will prove that this correspondance is one-to-one in~\cite{cc:isorel}.

To prove Theroem~\ref{cone}, it remains to show that the Handel-Miller foliation cones have disjoint interiors, which we prove in Theorem~\ref{max0},  and that the Handel-Miller foliation cones are finite in number, which  follows as before from the finiteness of the cones for $K_{0}$.

We have seen that $\intr\C_{\FF}$ consists of classes of foliated forms.  Any such form in the integer lattice restricts to an element of the integer lattice in each $\hI(K_{i})$,  hence has as period group a subgroup of $\Z$.  Thus the rational rays in $\intr\C_{\FF}$  correspond to  fibrations of $W$.

\begin{prop}\label{samecone'}
A rational foliated ray $\<\GG\>$ lies in $\intr\C_{\FF}$ if and only if $\C_{\GG}=\C_{\FF}$.
\end{prop}

\begin{proof}
Suppose that $\<\GG\>\subset\intr\C_{\FF}$. By Theorem~\ref{folforms}, we choose $\GG$ transverse to $\LL_{h}$. 
The proof of Proposition~\ref{samecone} carries through for $\<\GG|K_{i}\>$ and $\Ck_{\FF|K_{i}}$, $0\le i<\infty$, showing that  $\LL_{h}$ induces Handel-Miller monodromy for  $\GG|K_{i}$ also, $0\le i<\infty$.  In particular, the same family $\mu_{1},\dots,\mu_{r}$ defines the Handel-Miller foliation cone for both $\FF|K_{i}$ and $\GG|K_{i}$, $0\le i<\infty$, hence  $\C_{\FF}=\C_{\GG}$.  Again, the converse is trivial. 
\end{proof}

\begin{cor}\label{notinbd'}
No rational foliated ray is contained in $\bd\C_{\FF}$.
\end{cor}

\begin{proof}

This follows from Lemma~\ref{uniondense} and Proposition~\ref{samecone'} exactly as Corollary~\ref{notinbd} follows from Corollary~\ref{ratdense} and Proposition~\ref{samecone}.
\end{proof}

The boundary $\bd\C_{\FF}$ is made up of  $r$ top faces $F_{1},\dots,F_{r}$, where $F_{i}$ is a convex, polyhedral cone with nonempty (relative) interior in the hyperplane $[ \mu_{i}]=0$.

\begin{lemma}\label{denseinFi0}
Each $F_{i}$ contains a dense family of rays that meet nontrivial points of the integer lattice $\hI(\wh W;\Z)$.  
\end{lemma}

\begin{proof}

This is proven analogously to Corollary~\ref{uniondense}. Simply add the constraint $[ \mu_{i}]=0$.
\end{proof}

\begin{theorem}\label{max0}
If $g$ is a monodromy map \upn{(}endperiodic\upn{)} for a fibration $\GG$ of $W$, then either $(\intr\C_{g})\cap\C_{\FF}=\0$, or $\C_{g}\subseteq\C_{\FF}$.  In particular, $\C_{\FF}=\C_{h}$ is the maximal foliation cone for monodromies in the isotopy class of $h$  and two maximal foliation cones either coincide or have disjoint interiors.
\end{theorem}

\begin{proof}
If $(\intr\C_{g})\cap\C_{\FF}\ne\0$ and $\C_{g}\not\subseteq\C_{\FF}$, then Lemma~\ref{denseinFi0} implies that there is a rational foliated ray in $\bd\C_{\FF}$, contradicting Corollary~\ref{notinbd'}.
\end{proof}

\bibliographystyle{amsplain}

\begin{thebibliography}{10}


\bibitem{bca}
S.~A. Bleiler and A.~J. Casson, \emph{Automorphisms of surfaces after {N}ielsen
  and {T}hurston}, Cambridge Univ. Press, Cambridge, 1988.

\bibitem{condel1}
A.~Candel and L.~Conlon, \emph{Foliations, {I}}, American Mathematical Society,
  Providence, Rhode Island, 1999.

\bibitem{cc:almostnohol}
J.~Cantwell and L.~Conlon, \emph{Cones of foliations almost without holonomy},
  to appear in Geometry and Foliations 2013 -- Advanced Studies in Pure
  Mathematics.

\bibitem{cc:hm}
\bysame, \emph{Endperiodic automorphisms of surfaces and foliations},\hfill\break
 MathArXiv:1006:4525v5.


\bibitem{cc:isorel}
\bysame, \emph{Isotopies in open, foliated sets without holonomy} (in preparation).


\bibitem{cc:cone}
\bysame, \emph{Foliation cones}, Geometry and Topology Monographs, Proceedings
  of the Kirbyfest, vol.~{\bf 2}, 1999, pp.~35--86.

\bibitem{cc:LB}
\bysame, \emph{Isotopies of foliated $3$-manifolds without holonomy}, Adv. in
  Math. \textbf{{\bf 144}} (1999), 13--49.

\bibitem{cc:dum}
\bysame, \emph{Endsets of exceptional leaves; a theorem of {G}. {D}uminy},
  Proceedings of the {E}uroworkshop on Foliations, Geometry and Dynamics, World
  Scientific, 2002, pp.~225--261.

\bibitem{cc:norm}
\bysame, \emph{The sutured {T}hurston norm}, Foliations 2012, 2013, pp.~41--66.


\bibitem{deRham}
G.~de~Rham, \emph{Vari\'et\'es {D}ifferentiables}, Hermann, Paris, 1960.

\bibitem{dip}
P.~Dippolito, \emph{Codimension one foliations of closed manifolds}, Ann. of
  Math. \textbf{{\bf 107}} (1978), 403--453.


\bibitem{Epstein:isotopy}
D.~B.~A. Epstein, \emph{Curves on $2$-manifolds and isotopies}, Acta Math.
  \textbf{{\bf 115}} (1966), 83--107.

\bibitem{Ep:3}
D.~B.~A. Epstein, \emph{Periodic flows on $3$--manifolds}, Ann. of Math.
  \textbf{{\bf 95}} (1972), 66--82.

\bibitem{FLP}
A.~Fathi, F.~Laudenbach, and V.~Po\'enaru, \emph{Travaux de {T}hurston sur les
  {S}urfaces {(}2de \'ed.{)}}, Ast\'erisque \textbf{{\bf 66--67}} (1991).

\bibitem{fried}
D.~Fried, \emph{Fibrations over ${S}^{1}$ with pseudo-{A}nosov monodromy},
  Ast\'erisque \textbf{{\bf 66-67}} (1991), 251--266.

\bibitem{ga0}
D.~Gabai, \emph{Foliations and genera of links}, Topology \textbf{{\bf 23}}
  (1984), 381--394.


\bibitem{sue:closed}
S.~E. Goodman, \emph{Closed leaves in foliations of codimension one}, Comment.
  Math. Helv. \textbf{{\bf 50}} (1975), 383--388.


\bibitem{ham:disk-holes}
M.~E. Hamstrom, \emph{Some global properties of the space of homeomorphisms on
  a disk with holes}, Duke Math. J. \textbf{{\bf 29}} (1962), 657--662.

\bibitem{ham:torus}
\bysame, \emph{The space of homeomorphisms on a torus}, Ill. J. Math.  \textbf{{\bf 9}} (1965), 59--65.

\bibitem{ham:homeo}
\bysame, \emph{Homotopy groups of the space of homeomorphisms on a  $2$--manifold}, Ill. J. Math. \textbf{{\bf 10}} (1966), 563--573.

\bibitem{HandT}
M.~Handel and W.~P. Thurston, \emph{New proofs of some results of {N}ielsen},  Adv. in Math. \textbf{{\bf 56}} (1985), 173--191.

\bibitem{hatch}
A.~Hatcher, \emph{{Algebraic Topology}}, Cambridge University Press,
  Cambridge,UK, 2002.


\bibitem{plante:meas}
J.~F. Plante, \emph{Foliations with measure preserving holonomy}, Ann. of Math.
  \textbf{{\bf 102}} (1975), 327--361.


\bibitem{sa:pseudo}
R.~Sacksteder, \emph{Foliations and pseudogroups}, Amer. J. Math. \textbf{{\bf
  87}} (1965), 79--102.

\bibitem{schwartz}
L.~Schwartz, \emph{Th\'eorie des {D}istributions. nouvelle edition}, Hermann,
  Paris, 1966.

\bibitem{sch_cycles}
S.~Schwartzmann, \emph{Asymptotic cycles}, Ann. of Math. \textbf{{\bf 66}}
  (1957), 270--284.

\bibitem{sull:cycles}
D.~Sullivan, \emph{Cycles for the dynamical study of foliated manifolds and
  complex manifolds}, Inv. Math. \textbf{{\bf 36}} (1976), 225--255.

\bibitem{th:norm}
W.~Thurston, \emph{A norm on the homology of three-manifolds}, Mem. Amer. Math.
  Soc. \textbf{{\bf 59}} (1986), 99--130.

\bibitem{th:surfaces}
\bysame, \emph{On the geometry and dynamics of diffeomorphisms of surfaces},
  Bull. Amer. Math. Soc. \textbf{\bf 19} (1988), 417--431.


\bibitem{yag}
T.~Yagasaki, \emph{Homotopy types of homeomorphism groups of noncompact
  $2$-manifolds}, Top. and its Applications \textbf{\bf 108} (2000), 123--136.

\end{thebibliography}

\end{document}